\documentclass[journal]{IEEEtran}
\IEEEoverridecommandlockouts

\usepackage{cite}
\usepackage{amsmath,amssymb,amsfonts,amsthm}
\usepackage{hyperref}
\hypersetup{hidelinks}
\usepackage{algorithmic}
\usepackage{textcomp}
\usepackage{xcolor}
\usepackage{mathtools}
\usepackage{enumerate}
\usepackage{standalone}
\usepackage{graphicx}
\usepackage{epstopdf}
\usepackage{tabularray}
\usepackage{array}
\usepackage{subcaption}

\usepackage{tikz}

\usetikzlibrary{patterns}
\def\BibTeX{{\rm B\kern-.05em{\sc i\kern-.025em b}\kern-.08em
    T\kern-.1667em\lower.7ex\hbox{E}\kern-.125emX}}

\usetikzlibrary{positioning}
\usetikzlibrary{calc}
\usetikzlibrary{arrows,decorations.pathmorphing,backgrounds,positioning,fit,petri}
\usetikzlibrary {patterns,patterns.meta}

\newlength{\minimumblockheight}
\newlength{\minimumblockwidth}
\usetikzlibrary{shapes}

\tikzset{
block/.style = {draw, fill=white, rectangle, minimum height=2em, minimum width=3em},
sum/.style= {draw, fill=white, circle, node distance=1cm},
input/.style = {coordinate},
output/.style= {coordinate}
}

\begin{document}
\newcommand{\realpart}{\mathrm{Re}}
\newcommand{\imagepart}{\mathrm{Im}}
\newcommand{\diag}{\mathrm{diag}}
\newcommand{\trace}{\mathrm{Tr}}
\newcommand{\range}[1]{\mathcal{R}(#1)}
\newcommand{\kernel}[1]{\mathcal{K}(#1)}

\newcommand{\halfdiskmat}{\mathcal{S}_\gamma(\pi/2)}
\newcommand{\secmat}{\mathcal{S}_\gamma(\alpha,\beta)}
\newcommand{\secmatsym}{\mathcal{S}_\gamma(\alpha)}
\newcommand{\nsecmat}{\mathcal{SN}_\gamma(\alpha,\beta)}
\newcommand{\nsecmatsym}{\mathcal{SN}_\gamma(\alpha)}
\newcommand{\secsys}{\mathcal{U}_\gamma(\alpha,\beta)}
\newcommand{\secsyssym}{\mathcal{U}_\gamma(\alpha)}
\newcommand{\Hr}{\mathcal{H}_\gamma}
\newcommand{\Va}{\mathcal{V}_\alpha}
\newcommand{\Vhf}{\mathcal{V}_{\pi/2}}
\newcommand{\Kk}{\mathcal{K}_\delta}
\newcommand{\Ko}{\mathcal{K}_{\gamma\sec^2(\alpha)}}
\newcommand{\Ki}{\mathcal{K}_{\gamma\sec(\alpha)}}

\newcommand{\Mn}{\mathbb{C}^{n \times n}}
\newcommand{\Wn}{\mathbb{W}_n}
\newcommand{\RH}{\mathcal{RH}_\infty}
\newcommand{\RHn}{\RH^{n\times n}}
\newcommand{\Msec}[4]{\begin{bmatrix}{#1}&{#2}\\{#3}&{#4}\end{bmatrix}}

\newcommand{\DWUsubset}{D_1}
\newcommand{\DWUsuperset}{D_2}

\newtheorem{thm}{Theorem}
\newtheorem{cor}{Corollary}
\newtheorem{lem}{Lemma}
\newtheorem{prop}{Proposition}
\newtheorem{conj}{Conjecture}
\newtheorem{prob}{Problem}
\newtheorem{rem}{Remark}
\newtheorem{expl}{Example}

\title{Feedback Stability Under Mixed Gain and \\ Phase Uncertainty
}

\author{Jiajin Liang, Di Zhao, and Li Qiu, \IEEEmembership{Fellow, IEEE}
\thanks{This work was supported in part by Hong Kong Research Grants Council under GRF 16201120, by Shenzhen Science and Technology Innovation Committee under Shenzhen-Hong Kong-Macau Science and Technology Program (Category C) SGDX20201103094600006, by National Natural Science Foundation of China under grant 62103303 and 62088101, and by Googol Technology. }
\thanks{J. Liang is with the Googol Technology, Shenzhen, Guangdong, China 
(e-mail: jliangau@connect.ust.hk). }
\thanks{D. Zhao is with the Department of Control Science and Engineering, 
National Key Laboratory of Autonomous Intelligent Unmanned Systems, and Frontiers Science Center for Intelligent Autonomous Systems, Tongji
University, Shanghai, China (e-mail: dzhao925@tongji.edu.cn). }
\thanks{L. Qiu is with the Department of Electronic and Computer Engineering,
The Hong Kong University of Science and Technology, Clear Water Bay, Kowloon, Hong Kong 
(e-mail: eeqiu@ust.hk). }
}
\maketitle

\begin{abstract}
    In this study, we investigate the robust feedback stability problem for multiple-input-multiple-output linear time-invariant systems involving sectored-disk uncertainty, namely, dynamic uncertainty subject to simultaneous gain and phase constraints. This problem is thereby called a sectored-disk problem. 
    Employing a frequency-wise analysis approach, we derive a fundamental static matrix problem that serves as a key component in addressing the feedback stability.
    The study of this matrix problem heavily relies on the Davis-Wielandt (DW) shells of matrices, providing a profound insight into matrices subjected to simultaneous gain and phase constraints. This understanding is pivotal for establishing a less conservative sufficient condition for the matrix sectored-disk problem, from which we formulate several robust feedback stability conditions against sectored-disk uncertainty.
    Finally, several conditions based on linear matrix inequalities are developed for efficient computation and verification of feedback robust stability against sectored-disk uncertainty.
\end{abstract}


\section{Introduction}
\IEEEPARstart{R}{obust} control theory is dedicated to the study of system stability and performance in the presence of uncertainties arising from mathematical model variations and communication processes. In many applications, the uncertainty is characterized by a set of dynamical systems 
with certain gain constraints or simply norm bounds.
The salient small gain theorem gives a necessary and sufficient condition for 
robust stability of a closed-loop linear time-invariant (LTI) system with gain-bounded uncertainties.
A natural question is repeatedly asked: 
how to describe uncertainty in a more sophisticated and accurate way?
In regard of the single-input-single-output (SISO) LTI framework, 
a natural idea is to utilize phase information together with the gain.
While a widely accepted definition of ``phase'' for multiple-input-multiple-output (MIMO) LTI systems was elusive for an extended period, several endeavors were undertaken to formulate such a definition. These efforts aimed to harness more information for analyzing feedback stability and performance, as documented in works such as \cite{postlethwaite1981principal,owens1984numerical,anderson1988hilbert,haddad1992there,chen1998multivariable}.
Recently, a new and proper definition for the phase of MIMO LTI systems has emerged~\cite{wang2020phases,chen2021phase,ZhaoDi2022LowPhaseRank}, encompassing well-established positive-real and negative-imaginary notions. This development has led to the establishment of a novel phase theory~\cite{Chen2020,MaoXin2022DiscreteLtiPhase,Qiu2022NewPhase}.
In the realm of dynamical systems featuring phase-bounded uncertainty, the small phase theorem provides a sufficient condition for closed-loop robust stability~\cite{MaoXin2022DiscreteLtiPhase,Qiu2022NewPhase}.

\begin{figure}[h]
    \begin{subfigure}{0.49\linewidth}
        \centering
        \includegraphics{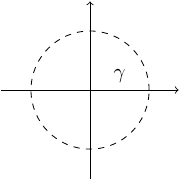}
        \caption{Single-disk constraint:\\ $r\leq \gamma$.}
    \end{subfigure}
    \begin{subfigure}{0.49\linewidth}
        \centering
        \includegraphics{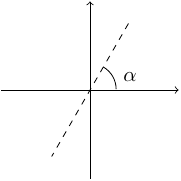}
        \caption{Half-plane constraint:\\ 
        $\theta\in[\alpha-\pi,\alpha]$.}
    \end{subfigure}
    \begin{subfigure}{0.49\linewidth}
        \centering
        \includegraphics{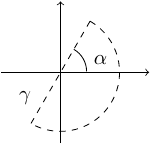}
        \caption{Half-disk constraint: \\$r\leq\gamma, \theta\in[\alpha-\pi,\alpha]$.}
    \end{subfigure}
    \begin{subfigure}{0.49\linewidth}
        \centering
        \includegraphics{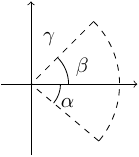}
        \caption{Sectored-disk constraint: \\
        $r\leq\gamma$, $\theta\in[\alpha,\beta]$.}
    \end{subfigure}
    \caption{Illustration of quadratic constraints on the uncertainty. The constraints are given in polar coordinates, i.e., a point on the plane is given by $c=re^{j\theta}$.}\label{fig:uncertainties}
\end{figure}

The robust control theory addressing gain-bounded uncertainty has undergone extensive exploration since Zames~\cite{Zames1966} introduced the pivotal small gain theorem for gain-bounded uncertain systems.
The robust stability problem, involving uniform gain-bounded uncertainty, can be conceptualized as a single-disk problem. In the case of a SISO system, the set of uncertain systems can be represented by a disk on the complex plane, as depicted in Fig.~\ref{fig:uncertainties}(a).
An extension of the single-disk problem is the multiple-disk problem, where uncertainties are modeled as block-diagonal transfer matrices, with each block norm-bounded by a given constant. Addressing the multiple-disk problem, the $\mu$-analysis technique~\cite{zhou1996robust} was introduced. This technique generalizes the singular value to the structural singular value (SSV), allowing for the derivation of an exact robust stability condition. However, computing the SSV is generally NP-hard, except for specific feedback systems with special structures~\cite{ZhaoDi2020TwoPort}.

In the frequency-domain analysis of SISO LTI systems, both gain and phase play crucial roles, exemplified by well-known tools such as the Bode plot and Nyquist plot. For MIMO systems, a convenient approach to integrating gain and phase-type information is through solving what is known as a half-disk problem.
As the name suggests, the half-disk problem extends the single-disk problem by introducing a passive or positive-real constraint on the uncertainty set. This constraint geometrically removes half of the disk on the complex plane. Eszter and Hollot~\cite{Eszter1994RobustnessUC} applied the $\mu$-analysis technique to address the half-disk problem. They derived an upper bound of the structural singular value (SSV) to establish a sufficient condition for robust stability, providing a geometric interpretation based on the generalized numerical range.
With the recent introduction of the MIMO phase concept, it is now possible to integrate gain and phase information for a more comprehensive stability analysis. Particularly, a positive-real or passive LTI system can be viewed as a system with phases ranging between $[-\pi/2, \pi/2]$. A noteworthy extension of the half-disk problem arises, termed the sectored-disk problem. This problem involves uncertainties with a single gain constraint and two constraints related to half-planes, specifically, a set of dynamical systems with gain less than $\gamma$ and phases within $[\alpha, \beta]$.
In the SISO case, the Nyquist plot of such uncertainty is contained within a sector-shaped area on the complex plane, providing motivation for the nomenclature.

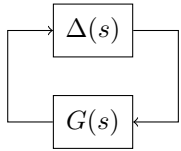
\begin{figure}[ht]
    \centering
    \begin{tikzpicture}[auto]
        \node[block] (plant) {$G(s)$};
        \node[block,above=0.5cm of plant] (b) {$\Delta(s)$};
        \draw[->] (b.0) -- ($(b.0)+(0.6,0)$) -- ($(plant.0)+(0.6,0)$) -- (plant.0);
        \draw[->] (plant.180) -- ($(plant.180)-(0.6,0)$) -- ($(b.180)-(0.6,0)$) -- (b.180);
    \end{tikzpicture}
    \caption{Uncertain feedback system}\label{fig_uncertain_system}
\end{figure}

The primary contribution of this study lies in proposing a robust stability condition for LTI systems involving sectored-disk uncertainty. Consider a closed-loop LTI system depicted in Fig.~\ref{fig_uncertain_system}, where $\Delta(s)$ represents a sectored-disk uncertain system—meaning it belongs to a set of norm and phase-bounded transfer functions.
From frequency-domain analysis, the robust stability of this closed-loop system is tantamount to the frequency-wise invertibility of $I+G(j\omega)\Delta(j\omega)$ when both $G$ and $\Delta$ are stable. Addressing this stability problem hinges on ensuring the invertibility of the matrix $I+AB$, where matrix $B$ spans a set that is both norm and phase-bounded. Termed the matrix sectored-disk problem, the numerical range of such a matrix lies within a sector on the complex plane.
To tackle the matrix sectored-disk problem with a graphical and less conservative approach, we leverage the notion of a higher-dimensional generalization of the numerical range—the Davis-Wielandt (DW) shell~\cite{Li2008DAVISWIELANDTSO}. Specifically, we derive a sufficient condition for the matrix sectored-disk problem by analyzing and approximating the DW shell union of all matrices within the sectored-disk uncertainty set.

\begin{figure}
        \centering
        \includegraphics[width=\linewidth]{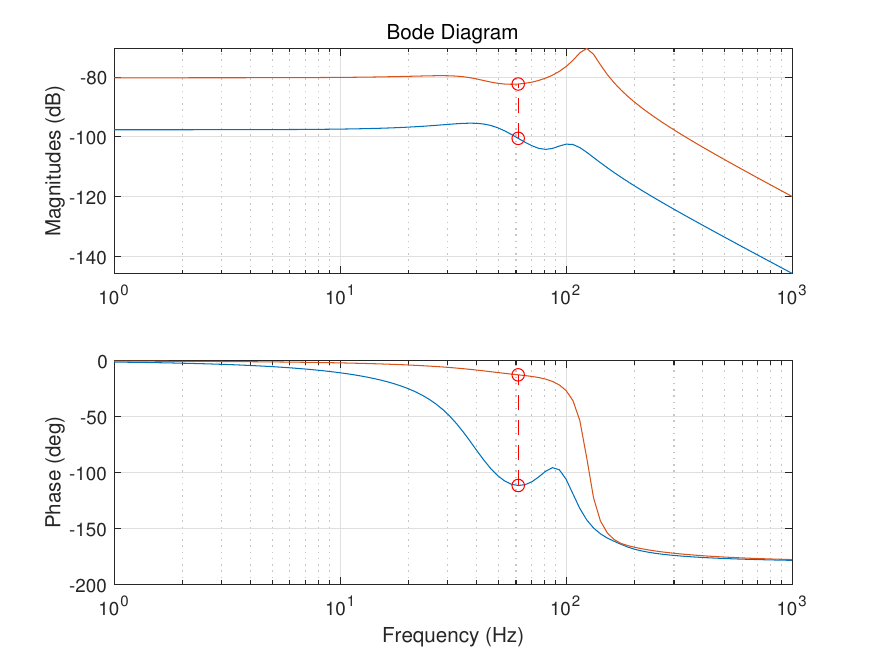}
        \caption{Bode diagram of $\Delta(s)$, and the gain and phase range at some fixed frequency.}
        \label{fig:colocated_system_bode}
\end{figure}

This study is primarily motivated by the frequently observed limitation of the small gain theorem and the small phase theorem, particularly in cases of combined gain and phase uncertainty, referred to as a sectored-disk problem. In certain scenarios, these stability conditions can be overly conservative.
For instance, in the context of designing controllers for some multivariable resonant system, the plant under consideration is described by the following transfer function:
\begin{align*}
    P(s) = 
    \underbrace{\sum\limits_{i=0}^{m}\frac{\psi_i \psi_i^*}{s^2+2\zeta_i\omega_is+\omega_i^2}}_{P_m(s)}
    +\underbrace{\sum\limits_{i=m+1}^{M}\frac{\psi_i \psi_i^*}{s^2+2\zeta_i\omega_is+\omega_i^2}}_{\Delta(s)}.
\end{align*}

Here, $M$ is a large integer, $\omega_i$ is the $i$-th natural frequency, $\zeta_i$ and $\psi_i$ are the corresponding damping ratio and mode vector.
Then plant $P(s)$ is usually approximated by a lower order transfer function $P_m(s)$, where $m>0$ is much smaller than $M$, and the residual is denoted by $\Delta(s)$. 
Previous studies have demonstrated that under mild conditions, $\Delta(s)$ is both norm-bounded and negative-imaginary. 
This implies that $\Delta(j\omega)$ is subjected to simultaneous gain and phase constraints at each frequency $\omega$ as shown in Fig.~\ref{fig:colocated_system_bode}. 
The sectored-disk uncertainty provides a more nuanced description of the uncertain dynamics compared to utilizing positive-realness, negative-imaginariness or norm-boundedness. 
It enables a deeper investigation into the feedback control problem.
With standard treatment and a designed controller $C$ stabilizing the nominal plant $P_m$, the feedback system involving $P$ and $C$ can be rearranged, as shown in Fig.~\ref{fig_sample_plant}, where $\Delta$ represents the uncertain dynamics and $G$ is composed of $P_m$ and $C$. 
Given the gain and phase information of $\Delta$, the robust feedback stability of $\Delta$ and $G$ then translates into a sectored-disk problem. 
As will be shown in Section~\ref{Sectored-disk uncertainties}, both the small gain theorem and small phase theorem can offer robust stability conditions for such a problem,
but these conditions only use partial information of the uncertainty set, thereby introducing conservatism in robust analysis. 
Consequently, the primary focus of this study is on resolving sectored-disk problems with the least conservative methods.

\begin{figure}[ht]
    \centering
    \begin{tikzpicture}[auto, node distance=1.5cm,>=latex']
        \node [block] (plant) {$P_m(s)$};
        \node [block, below of=plant,node distance=1.2cm] (controller) {$C(s)$};
        \node [block, above of=plant,node distance=1.2cm] (delta) {$\Delta(s)$};
        \coordinate [left of=plant,node distance=1.3cm] (tmp1) {};
        \node [input, left of=plant,node distance=1.7cm] (input) {};
        \node [sum, right of=plant,node distance=1.3cm] (sum1) {};
        \node [output, right of=plant,node distance=1.7cm] (output) {};
        \draw [->] (tmp1) |- (delta);
        \draw [->] (delta) -| (sum1);
        \draw [->] (plant) -- (sum1);
        \draw [->] (sum1) -- (output) |- (controller);
        \draw [->] (controller) -| (input) -- (plant);
    
        \draw[dashed] ($(plant.0)+(1.5,0.6)$) -- ($(controller.0)+(1.5,-0.8)$) node [left=0.5, above] {$G$} -- ($(controller.180)+(-1.5,-0.8)$) -- ($(plant.180)+(-1.5,0.6)$) -- ($(plant.0)+(1.5,0.6)$);
        \end{tikzpicture}
    \caption{Controlled feedback system}\label{fig_sample_plant}
\end{figure}
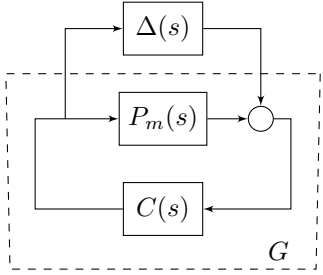

To mitigate conservatism in robust stability conditions, several attempts have been made to incorporate gain and phase-type information into the analysis. We now review some notable results and draw comparisons with our primary findings.
Tits et al.~\cite{Tits1999RobustnessUB} were pioneers in addressing the robust stability problem under structured bounded uncertainty, emphasizing the inclusion of phase-type information. They introduced the phase-sensitive structured singular value (PS-SSV) and proposed a sufficient condition for robust stability based on PS-SSV. However, it's noteworthy that computing the PS-SSV is an NP-hard problem. Even the proposed upper bound of PS-SSV necessitated solving a quasi-convex optimization problem.
In addition to the $\mu$-analysis method, Scherer~\cite{scherer2006LmiRelaxations} employed a relaxation technique on linear matrix inequalities (LMIs) for the sectored-disk uncertainty set. The study provided numerical results to assess the accuracy of the relaxation approach.
Researchers have explored other approaches to integrate gain and phase information. For instance, Patra and Lanzon~\cite{patra2011stability} investigated 'mixed' small gain and negative-imaginary transfer functions. These transfer functions are either negative-imaginary or gain-bounded in distinct frequency intervals, leading to a robust stability condition derived for the interconnection system of two such transfer functions.
The robust synthesis problem under mixed gain and phase uncertainties has also been addressed, particularly in the SISO case. Haddad et al.\cite{Haddad2007ControllerSW} proposed a synthesis scheme for SISO LTI feedback systems, ensuring guaranteed phase and norm bounds. More recently, \cite{Zhao2022WhenSG} investigated the feedback stability of MIMO LTI systems with combined gain and phase information, utilizing direct weighting and geometric methods.
In comparison, the robust stability results proposed in this study encompass a broader range of system types than those in \cite{Haddad2007ControllerSW,patra2011stability}. Furthermore, our results provide less conservative sufficient conditions compared to those in \cite{scherer2006LmiRelaxations,Zhao2022WhenSG}.

The rest of this paper is organized as follows.
Section~\ref{Preliminaries} is about basic notations and tools.
Section~\ref{Matrix sectored-disk Problem} characterizes the shape of a DW shell union of sectored-disk matrices, and then gives a sufficient condition and a necessary condition for the matrix sectored-disk problem.
Section~\ref{Sectored-disk uncertainties} considers a closed-loop system with
sectored-disk uncertainty, from which a robust stability condition is derived.
The paper is concluded in Section~\ref{conclusion}.

\section{Preliminaries}\label{Preliminaries}
Let $\mathbb{R}^n$ be the $n$-dimensional real vector space and
$\mathbb{C}^n$ be the $n$-dimensional complex vector space.
Given a complex number $c\in\mathbb{C}$, 
its real and imaginary parts are denoted by $\realpart(c)$ and $\imagepart(c)$, respectively.
A nonzero complex number $c$ can also be represented in the polar form as $c=r e^{j\varphi}$
with $r>0$ and $\varphi$ taking values in a half open $2\pi$-interval, 
typically $[0,2\pi)$ or $(-\pi,\pi]$.
Here $r=|c|$ is called the modulus or the magnitude and 
$\varphi=\angle c$ is called the argument or the phase.

Denote the set of $n\times n$ complex matrices as $\Mn$.
Given a complex matrix $A\in\Mn$, 
denote its element in the $i$th row and $j$th column as $a_{ij}$.
Denote the diagonal matrix $D\in\Mn$ with diagonal elements $d_1,...,d_n$ by
$\text{diag}\{d_1,...,d_n\}$.
Denote the transpose, conjugate and conjugate (Hermitian) transpose of $A$ by 
$A'$, $\overline{A}$ and $A^*$, respectively.
The range space and kernel space of a matrix $A$ are denoted by $\range{A}$ and $\kernel{A}$, respectively.
A matrix $A$ can be written as $A=H(A)+jK(A)$ with $H(A) = \dfrac{1}{2}(A+A^*)$ and $K(A) = \dfrac{1}{2j}(A-A^*)$.
Here $H(A)$ is called the Hermitian part of $A$ and 
$K(A)$ the skew-Hermitian part of $A$.
A matrix $A$ is normal if $AA^*=A^*A$.
Given a matrix set $\mathcal{A}\subset\Mn$, 
denote the complementary set of $\mathcal{A}$ in $\Mn$ as $\mathcal{A}^c$.

\subsection{Numerical range and Davis-Wielandt shell}
The matrix phase definition is based on the numerical range.
The numerical range of $A\in\Mn$ is defined as
\begin{equation*}
    \mathcal{W}(A) \coloneqq \{(u^*H(A)u,u^*K(A)u):~u\in \mathbb{C}^n, u^*u = 1 \},
\end{equation*}
which is a compact convex subset of $\mathbb{R}^2$
and contains points $(\realpart(\lambda_i),\imagepart(\lambda_i)), i=1,\dots,n$, where $\lambda_i$ are the eigenvalues of $A$\cite{horn1991Topics}.
The conic hull of $\mathcal{W}(A)$ is called the angular numerical range of $A$, which is given by
\begin{align*}
    \mathcal{W}'(A) \coloneqq \{(\realpart(u^*Au),\imagepart(u^*Au)):~u\in \mathbb{C}^n \}\subset\mathbb{R}^2.
\end{align*}

Given $p$ Hermitian matrices $A_1,\dots,A_p\in\Mn$, Fan and Tits~\cite{FanTits1987} defined a generalized numerical range as
\begin{multline*}
    \mathcal{W}(A_1,...,A_p) :=
    \{(u^*A_1u,\dots,u^*A_pu):\\
    u\in\mathbb{C}^n, u^*u=1\}\subset \mathbb{R}^p.
\end{multline*}
When $p=3$ and $n\geq 3$, the generalized numerical range is a convex set~\cite{FanTits1987}.
But the convexity is not guaranteed for $p>3$.
In~\cite{FanTits1988} and~\cite{FanTitsDoyle1991}, 
the generalized numerical range or m-form numerical range is linked to the computation of 
SSV and provides a geometric interpretation of SSV.

Given matrix $A\in \Mn$, the Davis-Wielandt shell (DW shell) of $A$ is defined as \cite{Li2008DAVISWIELANDTSO}
\begin{multline*}
    \mathcal{DW}(A) 
    \coloneqq \{(\realpart(u^*Au),\imagepart(u^*Au),u^*A^*Au):\\
    u\in \mathbb{C}^n, u^*u = 1 \} \subset \mathbb{R}^3.
\end{multline*}
A DW shell can be regarded as a generalized numerical range with $p=3$, namely,
$\mathcal{DW}(A) = \mathcal{W}(H(A),K(A),A^*A)$.

\begin{lem}[Section~2, \cite{Li2008DAVISWIELANDTSO}]\label{DW_shell_properties}
The DW shell of matrix $A\in\Mn$ has the following properties.
\begin{enumerate}
    \item $\mathcal{DW}(A)=\mathcal{DW}(U^*AU)$ for all unitary matrix $U\in\Mn$.
    \item $\mathcal{DW}(A)$ is a single point on a paraboloid for $n=1$.
    \item $\mathcal{DW}(A)$ is an ellipsoid without the interior for $n=2$.
    \item $\mathcal{DW}(A)$ is convex for $n\geq 3$.
    \item If $A\in\Mn$ is a normal matrix, 
    then $\mathcal{DW}(A)$ is the convex hull of points 
    $(\realpart(\lambda_i),\imagepart(\lambda_i),|\lambda_i|^2)$, 
    where $\lambda_i,i=1,...,n$ are the eigenvalues of $A$.
    \item $\mathcal{DW}(A) \subset \{(x,y,z)\in\mathbb{R}^3:~x^2+y^2\leq z\}$.
\end{enumerate}
\end{lem}

Given a set of matrices $\mathcal{A}\subset \Mn$, 
the DW shell union of matrices in $\mathcal{A}$ is denoted by 
$\mathcal{DW}(\mathcal{A})$, i.e.
\begin{equation*}
    \mathcal{DW}(\mathcal{A}) := \bigcup\limits_{A\in\mathcal{A}} \mathcal{DW}(A).
\end{equation*}

The notion of DW shell is useful to show the invertibility of matrix sum,
which is stated as the following lemma.
\begin{lem}[Theorem 2.1,\cite{li2008eigenvalues}]\label{lem_DW_shell_and_invertibility}
    Let $A,B\in\Mn$, then $\mathcal{DW}(-A)\cap \mathcal{DW}(B)\neq \emptyset$
    if and only if $\det(A+U^*BU)\neq 0$ for all unitary matrices $U\in\Mn$.
\end{lem}

\subsection{Matrix gain and phase definition}
Denote the $n$ singular values of matrix $A\in\Mn$ as
\begin{align*}
    \sigma(A)=
    \begin{bmatrix}
        \sigma_1(A) & \sigma_2(A) & \dots & \sigma_n(A)
    \end{bmatrix},
\end{align*}
with 
$\overline{\sigma}(A)=\sigma_1(A)\geq\sigma_2(A)\geq\dots\geq\sigma_n(A)=\underline{\sigma}(A)$.
The singular values of $A$ are regarded as the magnitudes of $A$. Denote the set of matrices whose norms are bounded by $\gamma$ as
\begin{align*}
    \mathcal{D}_\gamma := 
    \{
        A\in\Mn:~ \bar{\sigma}(A)\leq \gamma
    \}.
\end{align*}

In frequency-domain robust stability analysis, 
the problem of determining the invertibility of $I+AB$ plays a significant role.
More specifically, given a matrix $A$ and a set of matrices $\mathcal{B}$, 
we wish to determine if $\det(I+AB)\neq 0$ for all $B\in\mathcal{B}$.
The selection of matrix set $\mathcal{B}$ corresponds to the characterization of the uncertainty set.
In the single-disk problem, $\mathcal{B}$ is chosen to be $\mathcal{D}_\gamma$
and the matrix small gain theorem gives the necessary and sufficient condition.

\begin{lem}[Matrix small gain theorem~\cite{DoyleStein}]\label{lem:matrix_sg}
    Let $\gamma>0$ and $A\in\Mn$, then $\det(I+AB)\neq 0$ for all $B\in\mathcal{D}_\gamma$
    if and only if $\overline{\sigma}(A)<1/\gamma$.
\end{lem}

To define matrix phase, we start with the concept of sectorial matrix.
A matrix $A\in\Mn$ is called sectorial if $0\notin\mathcal{W}(A)$.

For a sectorial matrix $A\in\Mn$, 
it admits the following sectorial decomposition~\cite{zhang2015matrix},
\begin{align*}
    A = T^*DT,
\end{align*}
where $D=\diag(e^{j\theta_1},\dots,e^{j\theta_n})$ is a unitary diagonal matrix
and $T\in\Mn$ is nonsingular.
Here, $D$ is unique up to a permutation,
and we can assume that 
$\theta_1 \geq \theta_2 \geq \cdots \geq \theta_n$
and $\theta_1-\theta_n<\pi$.
Notice that these $\theta_i$ are determined modulo $2\pi$.
To uniquely determine $\theta_1,\dots,\theta_n$, 
denote $\gamma(A)=(\theta_1+\theta_n)/2$, and choose $\gamma(A)\in[-\pi,\pi)$.
Then $\gamma(A)$ is unique and is called the phase center of $A$.

The phases of $A$ are defined to be 
\[
\phi(A)=
\begin{bmatrix}
    \phi_1(A) & \phi_2(A) & \cdots & \phi_n(A)
\end{bmatrix},
\]
where
$\phi_i(A)=\theta_i, i=1,\ldots,n$.

Denote the interior and the boundary of $\mathcal{W}(A)$ 
by $\text{int}\mathcal{W}(A)$ and $\partial\mathcal{W}(A)$, respectively.
If $0\notin \text{int~}\mathcal{W}(A)$, then $A$ is called semi-sectorial.
The concept of phase can be extended to semi-sectorial matrices.
From~\cite{furtado2003spectral}, a semi-sectorial matrix $A$ has a decomposition 
\begin{align*}
    A = T^*
    \begin{bmatrix}
        0 \\ & D \\ & & E
    \end{bmatrix}
    T
\end{align*}
where $T$ is nonsingular, $D=\diag(e^{j\theta_1},\dots,e^{j\theta_m})$ and
$E$ is a direct sum of $k$ copies of the block
\begin{align*}
    e^{i\theta_0}
    \begin{bmatrix}
        1 & 2 \\ 0 & 1
    \end{bmatrix},
\end{align*}
with $k$ and $m$ satisfying $m+2k = \mathrm{rank}(A)$.
Again, we define the phase center of $A$ as $\gamma(A)=(\theta_1+\theta_m)/2$ such that $\gamma(A)\in[-\pi,\pi)$.
Then the phases of semi-sectorial matrix $A$, denoted by $\phi_i(A)$, $i=1,2,\dots,m+2k$, are defined as the non-increasing sequence composed of
$\theta_1,\dots,\theta_m$ and $k$ copies of $\gamma(A) \pm \pi/2$. The largest and smallest phases of $A$ are respectively denoted by
$\overline{\phi}(A)=\phi_1(A)$ and $\underline{\phi}(A)=\phi_{m+2k}(A)$.

When $k=0$, $A$ is congruent to a diagonal matrix and is called a quasi-sectorial matrix.

Denote the set of phase-bounded semi-sectorial matrices by
\begin{multline*}
\mathcal{S}(\alpha,\beta)=
\{
    A\in\Mn:~ A \text{ is semi-sectorial},\\
    [\underline{\phi}(A),\bar{\phi}(A)]\subset [\alpha,\beta]
\}
\end{multline*}
For $\alpha\in[0,\pi/2]$, we also denote that $\mathcal{S}(\alpha)=\mathcal{S}(-\alpha,\alpha)$.

When the given matrix $A$ and the uncertainty set are phase-bounded,
the following matrix small phase theorem provides 
a necessary and sufficient condition for the invertibility of $I+AB$ for all 
$B\in\mathcal{S}(\alpha,\beta)$.

\begin{lem}[Matrix small phase theorem, Lemma 4\cite{wang2020phases}]
    Let $A\in\Mn$ be a quasi-sectorial matrix with $\gamma(A)\in[-\pi,\pi)$,
    then $\det(I+AB)\neq 0$ for all $B\in\mathcal{S}(\alpha,\beta)$
    if and only if $[\alpha,\beta]\subset(-\pi-\underline{\phi}(A),\pi-\overline{\phi}(A))$ 
    modulo $2\pi$.
\end{lem}

\subsection{Gain and phase of MIMO LTI system}
Denote by $\mathcal{L}_\infty^{m\times m}$ the set of all $m \times m$ transfer matrices that are essentially bounded on the imaginary axis and by $\RH^{m\times m}$
the set of all $m \times m$ stable real-rational transfer matrices.
Given an LTI system $G\in\RH^{m\times m}$, 
denote the frequency response at $\omega$ as $G(j\omega)$, which is a constant complex matrix.
Then $\sigma(G(j\omega))$, the vector of singular values of $G(j\omega)$, 
is an $\mathbb{R}^m$-valued function of the frequency, 
which we call the magnitude response of $G$. 
The $\mathcal{H}_\infty$ norm of $G$ is denoted by 
\begin{align*}
    \|G\|_\infty \coloneqq \mathrm{ess }\sup_{\omega\in\mathbb{R}} \bar{\sigma}(G(j\omega)). 
\end{align*}

\vspace{-0.5cm}
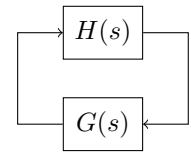
\begin{figure}[ht]
    \centering
    \begin{tikzpicture}[auto]
        \node[block] (plant) {$G(s)$};
        \node[block,above=0.5cm of plant] (b) {$H(s)$};
        \draw[->] (b.0) -- ($(b.0)+(0.6,0)$) -- ($(plant.0)+(0.6,0)$) -- (plant.0);
        \draw[->] (plant.180) -- ($(plant.180)-(0.6,0)$) -- ($(b.180)-(0.6,0)$) -- (b.180);
    \end{tikzpicture}
    \caption{Negative feedback System}\label{fig_negative_feedback_system}
\end{figure}

Given two LTI systems represented by $m\times m$ real-rational proper transfer matrices $G$ and $H$, consider the negative feedback system consisting of $G$ and $H$ as in Fig.~\ref{fig_negative_feedback_system}.
The closed-loop system is stable if the Gang of Four transfer matrix 
\begin{align*}
    G \# H := 
    \begin{bmatrix}
        (I+GH)^{-1} & (I+GH)^{-1}H \\
        G(I+GH)^{-1} & G(I+GH)^{-1}H
    \end{bmatrix}
\end{align*}
is stable, or $ G \# H\in\RH^{2m\times 2m}$.

Given a frequency function $\gamma(\omega)>0$,
consider a set of norm-bounded uncertain systems
$$\mathcal{B}_\gamma\coloneqq\{H\in\RHn :~ \overline{\sigma}(H(j\omega)) \leq \gamma(\omega), \forall \omega\in[0,\infty]\}.$$
A robust stability condition for $G\# H$ is given by the following small gain theorem.

\begin{lem}[System small gain theorem,\cite{DoyleStein}]
    Let $G\in\RHn$ and $\gamma(\omega)>0$, for all $\omega\in[0,\infty]$. 
    Then $G\# H$ is stable for all $H\in\mathcal{B}_\gamma$
    if and only if $\overline{\sigma}(G(j\omega)) < 1/\gamma(\omega)$ for all $\omega\in[0,\infty]$.
\end{lem}

If $G(j\omega)$ is sectorial (quasi-sectorial, semi-sectorial) for all $\omega \in [-\infty,\infty]$,
then the LTI system $G$ is called frequency-wise sectorial 
(quasi-sectorial, semi-sectorial).

System phase responses can be defined for frequency-wise semi-sectorial systems.
For a frequency-wise semi-sectorial system $G$, its DC phases are defined as $\phi(G(0))$.
Notice that $G(0)$ is a real matrix, the phase center of $G(0)$ is either 0 or $\pi$.
Without loss of generality, we can assume that $\phi(G(0))=0$, otherwise we can replace $G$ with $-G$.
Then the phase response of a given system $G$ at each frequency $\omega$ is defined as $\phi(G(j\omega))$.

For frequency-wise semi-sectorial system $G$, the maximum and minimum phases of $G$ are respectively defined as
\begin{align*}
    &\overline{\phi}(G) := \sup_{\omega\in[0,\infty]} \overline{\phi}(G(j\omega)), \\
    &\underline{\phi}(G) := \inf_{\omega\in[0,\infty]} \underline{\phi}(G(j\omega)).
\end{align*}
The $H_\infty$ phase sector of $G$ is defined as 
\[
    \Phi_\infty(G) := [\underline{\phi}(G), \overline{\phi}(G)],
\]
which can be regarded as the phase counterpart to the $\mathcal{H}_\infty$ norm of $G$.

With the newly defined system phases, let $\alpha(\omega),\beta(\omega)$ satisfying
$-\pi\leq(\alpha(\omega)+\beta(\omega))/2<\pi$ and $0<\beta(\omega)-\alpha(\omega)\leq\pi$. Define a cone of systems 
\begin{multline*}
    \mathcal{C}(\alpha,\beta)\coloneqq\{H\in\RHn :~ 
    H \text{ is frequency-wise semi}\\ \text{-sectorial, }
    \Phi_\infty(H(j\omega))\subset[\alpha(\omega),\beta(\omega)],
    \forall\omega\in[0,\infty]\}.
\end{multline*}
A robust stability condition for $G\# H$ is given by the following small phase theorem.

\begin{lem}[System small phase theorem,\cite{chen2021phase}]    
    Let $G\in\RHn$ be frequency-wise quasi-sectorial.
    $\alpha(\omega),\beta(\omega)$ satisfying
    $-\pi\leq(\alpha(\omega)+\beta(\omega))/2<\pi$ and $0<\beta(\omega)-\alpha(\omega)\leq\pi$.     
    Then $G \# H$ is stable for all $H\in\mathcal{C}(\alpha,\beta)$ if 
    \begin{align*}
        \overline{\phi}(G(j\omega)) < \pi - \beta(\omega), 
        \underline{\phi}(G(j\omega)) > -\pi - \alpha(\omega),
    \end{align*}
    for all $\omega\in[0,\infty]$.
\end{lem}

Comparing with the small gain theorem, the necessity of the small phase theorem is still open. For more discussions on the necessity direction, interested readers are referred to~\cite{chen2021phase}. 

\section{Matrix sectored-disk Problem}\label{Matrix sectored-disk Problem}

If a matrix is simultaneously norm and phase bounded, 
its numerical range is contained in a sector.
Such a matrix is called a sectored-disk matrix.
Given $\gamma>0$ and $\alpha,\beta$ satisfying $-\pi\leq(\alpha+\beta)/2<\pi$ and
$0<\beta-\alpha\leq\pi$,
denote the set of sectored-disk matrices with radius $\gamma$ and phase sector $[\alpha,\beta]$,  as illustrated in Fig.~\ref{fig:nr_sectored_disk}, by
\begin{multline*}
    \mathcal{S}_\gamma(\alpha,\beta) := \{A\in \Mn: A \text{ is semi-sectorial},\\
    [\underline{\phi}(A),\bar{\phi}(A)] \subset [\alpha,\beta], \bar{\sigma}(A) \leq \gamma \}.
\end{multline*}
Clearly, we have the following set relation: 
$$\mathcal{S}_\gamma(\alpha,\beta) = \mathcal{D}_\gamma \cap \mathcal{S}(\alpha,\beta).$$

For $\alpha\in[0,\pi/2]$, we also denote the set of symmetric sectored-disk matrices as
\begin{align*}
    \secmatsym &= \mathcal{S}_\gamma(-\alpha,\alpha).
\end{align*}

\begin{figure}[ht]
    \centering
    \includegraphics{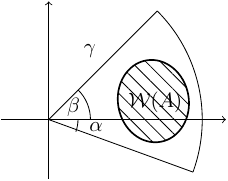}
    \caption{Numerical range and the sectored-disk area}
    \label{fig:nr_sectored_disk}
\end{figure}

The following matrix sectored-disk problem plays a significant role in 
frequency-domain robust stability analysis.
\begin{prob}[Sectored-disk problem]\label{sectored_disk_problem}
    Given matrix $A\in\Mn$ and $\alpha,\beta$ satisfying 
    $-\pi\leq(\alpha+\beta)/2<\pi$ and $0<\beta-\alpha\leq\pi$, 
    what is the requirement on $A$ such that $\det(I+AB)\neq 0$ for all $B\in \secmat$? 
\end{prob}
Notice that for $B\in\secmat$ and $\theta=-(\alpha+\beta)/2$, we have 
\begin{equation}\label{rotation_property}
    e^{j\theta}B\in\mathcal{S}_\gamma((\beta-\alpha)/2).
\end{equation}
Therefore, in most of the time throughout the paper we consider symmetric sectored-disk matrices $\secmatsym$ without loss of generality.
When $\alpha=\pi/2$, this problem is also called a half-disk problem.

In the scalar case, Problem~\ref{sectored_disk_problem} can be solved explicitly.
Since the proof is straight-forward, we just state the result without proof, interested readers are referred to \cite{Zhao2022WhenSG}.
\begin{lem}
    Given $a\in\mathbb{C}$ and 
    $\alpha,\beta$ satisfying $-\pi\leq(\alpha+\beta)/2<\pi$ and $0<\beta-\alpha\leq\pi$,
    then $1+ab\neq 0$ for all $b\in\secmat$ if and only if
    $|a|>1/\gamma$ or $\angle a\notin [-\pi-\alpha,\pi-\beta]$.
\end{lem}

In the following, we explore Problem~\ref{sectored_disk_problem} for $n\geq 2$.

\subsection{Connection with matrix DW shell}
For invertible matrix $A$, the invertibility of $(I+AB)$ is equivalent to that of $(A^{-1}+B)$.
A corollary of Lemma~\ref{lem_DW_shell_and_invertibility} can be used to tackle the matrix sectored-disk problem.
\begin{cor}\label{corollary_DW_shell_separation}
    Let $\gamma>0$, $-\pi\leq(\alpha+\beta)/2<\pi$,
    and $0<\beta-\alpha<\pi$. 
    Let $C\in\Mn$, then $\det(C+B)\neq 0$ for all $B\in\secmat$ if and only if
    \[\mathcal{DW}(-C) \cap \mathcal{DW}(\secmat) = \emptyset.\]
\end{cor}
\begin{proof}
    For sufficiency, let $B$ be any matrix in $\secmat$.
    From Lemma \ref{lem_DW_shell_and_invertibility},
    $\mathcal{DW}(-C) \cap \mathcal{DW}(B) = \emptyset$
    is equivalent to that $\det(C+U^*BU)\neq 0$ for any unitary matrix $U$.
    Let $U=I_n$, and it follows $\det(C+B)\neq 0$.

    For necessity, we know that $\det(C+B)\neq 0$ for any $B\in\secmat$. Suppose to the contrapositive that
    $\mathcal{DW}(-C) \cap \mathcal{DW}(B) \neq \emptyset$
    for some matrix $B\in\secmat$.
    From Lemma~\ref{lem_DW_shell_and_invertibility}, 
    there exists unitary matrix $U$ such that 
    $\det(C+U^*BU) = 0$.
    Let $\hat{B} = U^*BU$.
    Since unitary congruence does not change the phases and singular values of a matrix,
    we have $\hat{B}\in\secmat$.
    Unitary congruence also preserves the shape of DW shell. 
    So $\mathcal{DW}(-C) \cap \mathcal{DW}(\hat{B}) \neq \emptyset$ and 
    $\hat{B}$ satisfies $\det(C+\hat{B})=0$, contradicting to that $\det(C+B)\neq 0$ for any $B\in\secmat$.
\end{proof}

\subsection{DW shell union of sectored-disk matrices}
From Corollary~\ref{corollary_DW_shell_separation}, the problem of determining the invertibility of $I+AB$ is converted into the problem of determining the separation of two sets in $\mathbb{R}^3$.
Problem~\ref{sectored_disk_problem} is equivalent to the following problem when $A$ is invertible.
\begin{prob}\label{DW_shell_union_separation_problem}
    Given $\alpha\in(0,\pi/2)$, $\gamma>0$, and $A\in\Mn$ being invertible,
    check whether $\mathcal{DW}(-A^{-1})\cap\mathcal{DW}(\secmatsym)$ is empty.
\end{prob}

From Lemma~\ref{DW_shell_properties}, we observe that the DW shell of an $n \times n$ matrix is a convex set or the surface of a convex set. This convex set can be accurately characterized in $\mathbb{R}^3$ through numerical methods with arbitrary precision\cite{Lestas2012Large}. Therefore, once the shape of $\mathcal{DW}(\secmatsym)$ is determined, the matrix sectored-disk problem transforms into the verification of whether a convex set is separated from another set.

Utilizing the Monte Carlo sampling technique, we approximate the XZ-plane projection of $\mathcal{DW}(\mathcal{S}_1(\pi/3))$, as depicted in the grey area in Fig. \ref{DW_shell_union_numerical}.

\vspace{-0.5cm}
\begin{figure}[ht]
    \centering
    \includegraphics[width=0.66\columnwidth]{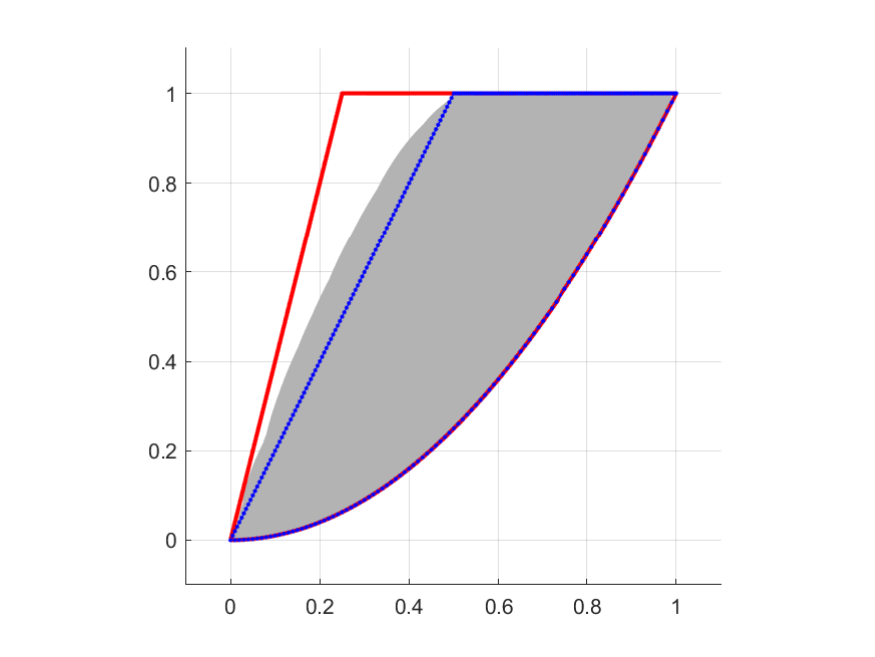}
    \caption{
        Monte Carlo simulation result of $\mathcal{DW}(\mathcal{S}_1(\pi/3))$.
        The gray area represents the XZ projection of the union of DW shells, computed for $3\times 10^9$ randomly generated 3-by-3 matrices. The blue line outlines the boundary of a subset within $\mathcal{DW}(\mathcal{S}_1(\pi/3))$, while the red line illustrates the boundary of a superset for $\mathcal{DW}(\mathcal{S}_1(\pi/3))$.}
    \label{DW_shell_union_numerical}
\end{figure}

Unfortunately, for a general $\alpha\in(0,\pi/2)$, the exact shape of $\mathcal{DW}(\secmatsym)$ remains elusive. To address Problem~\ref{DW_shell_union_separation_problem}, a practical approach involves identifying both a subset and a superset of $\mathcal{DW}(\secmatsym)$. Subsequently, the goal is to determine if these sets are separated from $\mathcal{DW}(-A^{-1})$. Specifically, assume that we have characterized two sets, $\DWUsubset$ and $\DWUsuperset$, satisfying:
\begin{align*}
    \DWUsubset \subset \mathcal{DW}(\secmatsym) \subset \DWUsuperset.
\end{align*}
Then for any matrix $C$, 
\begin{align*}
    &\mathcal{DW}(C)\cap\DWUsuperset=\emptyset \Rightarrow \quad
    \mathcal{DW}(C)\cap\mathcal{DW}(\secmatsym)=\emptyset, \\
    &\mathcal{DW}(C)\cap\mathcal{DW}(\secmatsym)=\emptyset \Rightarrow  \quad
    \mathcal{DW}(C)\cap\DWUsubset=\emptyset.
\end{align*}

Hence, the characterization of a superset $\DWUsuperset$ provides a sufficient condition for the assertion in Problem~\ref{DW_shell_union_separation_problem} to hold. Simultaneously, the characterization of a subset $\DWUsubset$ establishes a necessary condition. Notably, a smaller superset results in a more stringent sufficient condition, while a larger subset yields a more rigorous necessary condition. In the event that the superset and subset coincide, the condition becomes both necessary and sufficient.

Candidates for $\DWUsubset$ and $\DWUsuperset$ can be found directly from the properties of the DW shell.
Before proceeding, we define some useful subsets in $\mathbb{R}^3$ as 
\begin{align*}
    &\mathcal{P} = \{(x,y,z)\in\mathbb{R}^3:~x^2+y^2 \leq z \},\\
    &\mathcal{H}_\gamma = \{(x,y,z)\in\mathcal{P}:~z\leq \gamma^2\},\\
    &\mathcal{V}_\alpha = \{(x,y,z)\in\mathcal{P}:~-\tan(\alpha)x\leq y\leq\tan(\alpha)x\},\\
    &\mathcal{K}_k = \{(x,y,z)\in\mathcal{P}:~z\leq kx\}.
\end{align*}
When $\alpha=\pi/2$, the set $\mathcal{V}_{\pi/2}$ is given by $\{(x,y,z)\in\mathcal{P}:x\geq 0\}$.
The shapes of these sets are shown in Fig.~\ref{fig:set_shapes}.

\vspace{-0.5cm}
\begin{figure}[h]
    \begin{subfigure}{0.45\linewidth}
        \centering
        \includegraphics[width=\linewidth]{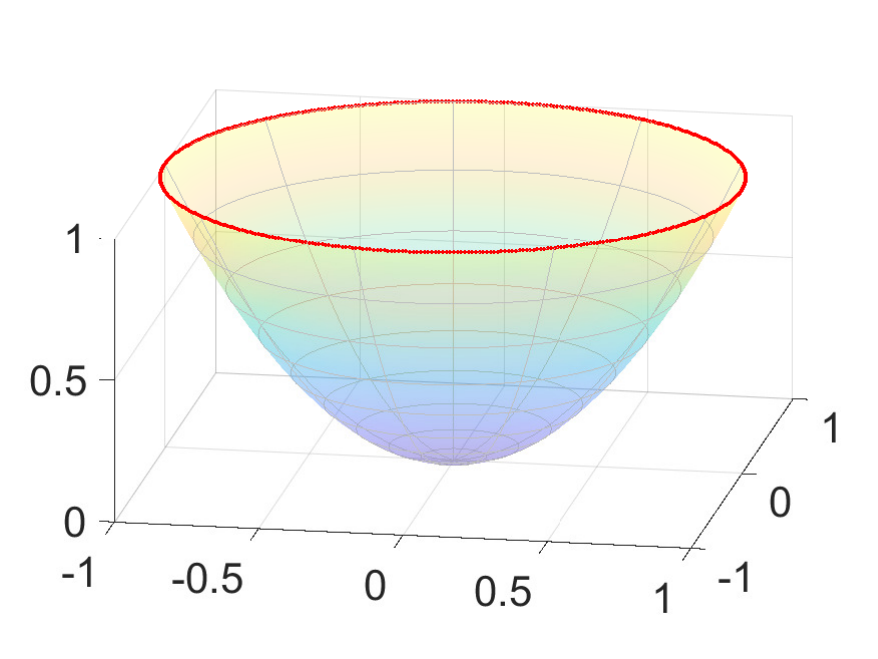}
        \caption{Shape of $\mathcal{H}_\gamma$}
        \label{fig:subim1}
    \end{subfigure}
    \begin{subfigure}{0.48\linewidth}
        \centering
        \includegraphics[width=\linewidth]{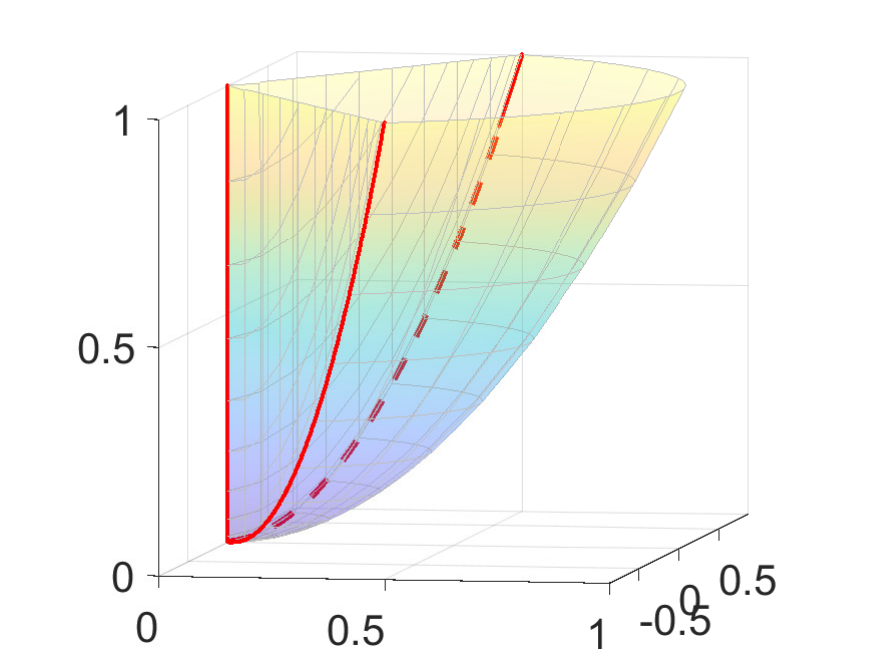}
        \caption{Shape of $\mathcal{V}_\alpha$}
        \label{fig:subim2}
    \end{subfigure}
    \begin{subfigure}{0.47\linewidth}
        \centering
        \includegraphics[width=\linewidth]{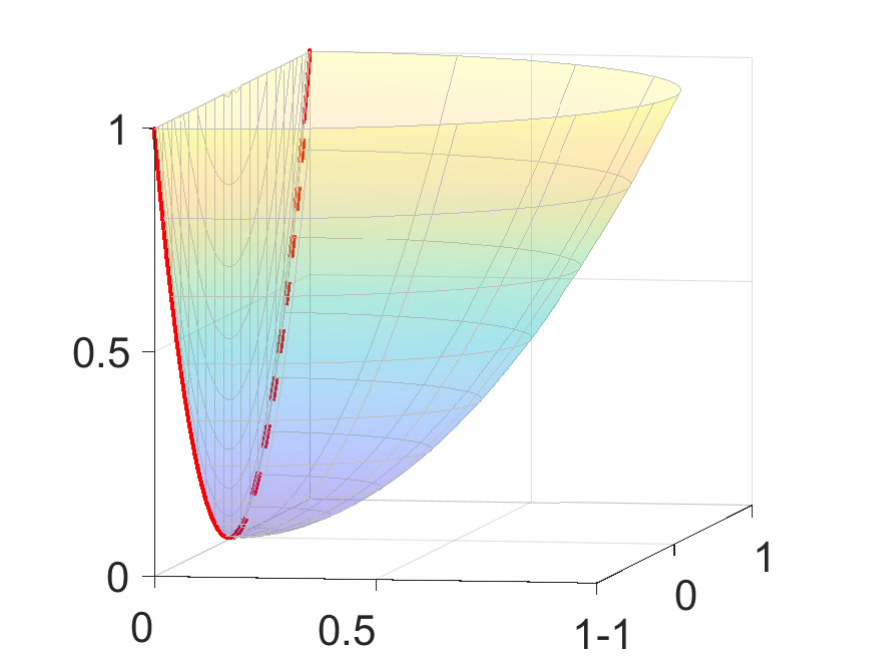}
        \caption{Shape of $\mathcal{V}_{\pi/2}$}
        \label{fig:subim3}
    \end{subfigure}
    \begin{subfigure}{0.49\linewidth}
        \centering
        \includegraphics[width=\linewidth]{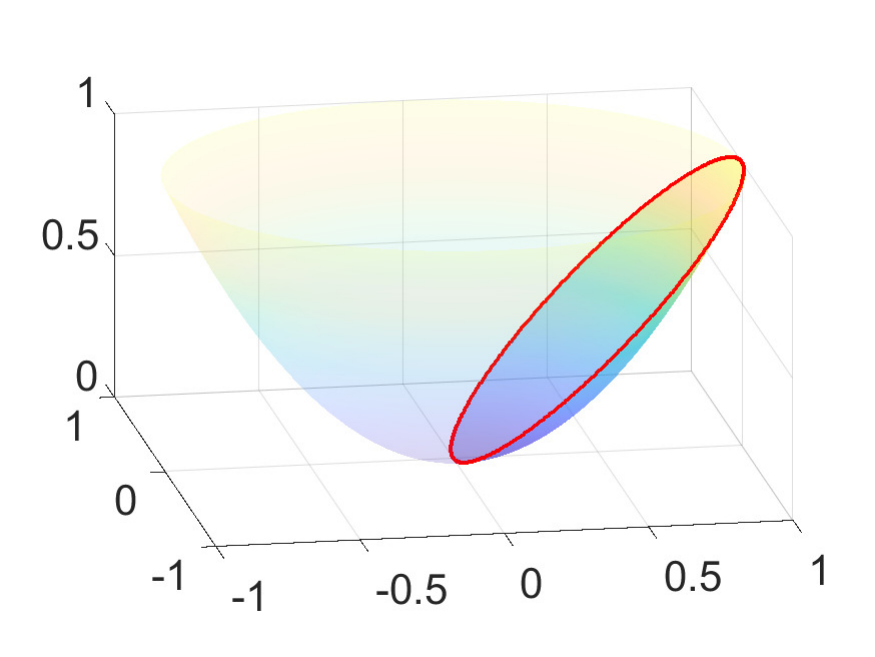}
        \caption{Shape of $\mathcal{K}_k$}
        \label{fig:subim4}
    \end{subfigure}
    \caption{Shapes of some useful sets in $\mathbb{R}^3$.}
    \label{fig:set_shapes}
\end{figure}

Denote the set of symmetric sectored-disk matrices that are in addition normal as
\begin{align*}
    \nsecmatsym = \{B\in\Mn :~ B\in\secmatsym, B \text{ is normal} \}.
\end{align*}
By Lemma~\ref{DW_shell_properties}, the DW shells of normal matrices are polytopes in $\mathbb{R}^3$.
The following proposition further characterizes the shapes and relations of the above sets.

\begin{prop}\label{prop_DW_shell_union_rough_shape}
    For $n\geq 3$, we have
    \begin{align}
        \mathcal{DW}(\nsecmatsym) &= \Hr \cap \Va \cap \Ki \label{normal_matrix_DW_shell_union}, \\
        \mathcal{DW}(\secmatsym) &\subset \Hr \cap \Va \label{sectored_disk_matrix_DW_shell_union}.
    \end{align}
\end{prop}
\begin{proof}
    See Appendix~\ref{proof_of_prop_DW_shell_union_rough_shape}.
\end{proof}

\begin{figure}
    \centering
    \includegraphics[width=.6\linewidth]{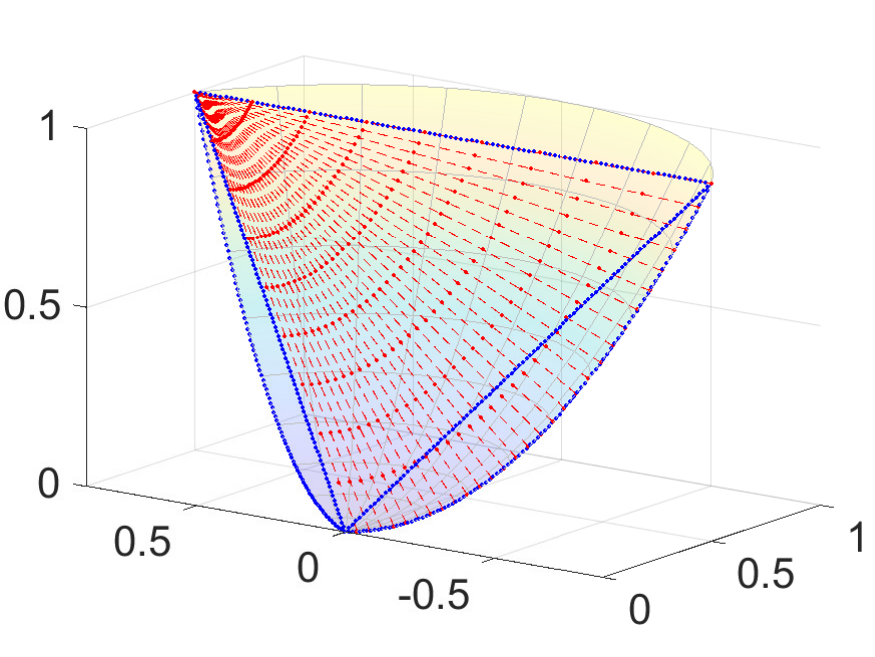}
    \caption{Envelop of $\mathcal{DW}(\nsecmatsym)$ for $n=2$.}
    \label{fig:dim2_normal_matrix_DW_union}
\end{figure}

Note that for $n=2$, equqlity \eqref{normal_matrix_DW_shell_union} does not hold in general. Instead, we observe that only the following relation holds:
\begin{equation}\label{dim_2_normal_matrix_DW_shell_union}
\mathcal{DW}(\nsecmatsym)\subsetneq \Hr \cap \Va \cap \Ki,
\end{equation}
This observation is further explained below. We assert that the intersection of $\mathcal{DW}(\nsecmatsym)$ and $\Ki$ is limited to the endpoints: $(\gamma\cos(\alpha),\gamma\sin(\alpha),\gamma^2)$, $(\gamma\cos(\alpha),-\gamma\sin(\alpha),\gamma^2)$, $(0,0,0)$, or a convex combination of two of them, forming a triangle as depicted in Fig.~\ref{fig:dim2_normal_matrix_DW_union}. However, it is important to note that the interior of this triangle is not contained within $\mathcal{DW}(\nsecmatsym)$.
This assertion is supported by considering a point $p\in\mathcal{DW}(\nsecmatsym)\cap\Ki$, which can be expressed as the convex combination of two points
$p_1=(x_1,y_1,x_1^2+y_1^2)$ and $p_2=(x_2,y_2,x_2^2+y_2^2)$ 
according to Lemma~\ref{DW_shell_properties}. From the constraints imposed by sectored-disk matrices and the convexity of the paraboloid, it follows that $x_i^2+y_i^2\leq\gamma\sec(\alpha)x_i$ for $i=1,2$. The equality only holds for the three points $(\gamma\cos(\alpha),\gamma\sin(\alpha),\gamma^2)$, $(\gamma\cos(\alpha),-\gamma\sin(\alpha),\gamma^2)$, $(0,0,0)$ and the line segments connecting two of the endpoints, leading to \eqref{dim_2_normal_matrix_DW_shell_union}.

When $\alpha=\pi/2$, the superset and subset in Proposition~\ref{prop_DW_shell_union_rough_shape} coincide.
The following corollary characterizes the DW shell union of half-disk matrices.
\begin{cor}\label{cor_half_disk_union_shape}
    It holds the following equality
    $$\mathcal{DW}(\halfdiskmat)=\Hr\cap\mathcal{V}_{\pi/2}.$$
\end{cor}

As discussed earlier, a superset of the DW shell union corresponds to a sufficient separating condition, while a subset corresponds to a necessary condition. Given that $\mathcal{DW}(\halfdiskmat)$ is precisely characterized, the DW shell separating condition emerges as both necessary and sufficient, and it will be presented subsequently.

Proposition~\ref{prop_DW_shell_union_rough_shape} delineates $\Hr\cap\Va$ as a superset of $\mathcal{DW}(\secmatsym)$. This superset, stemming from the gain and phase constraints, tends to overestimate $\mathcal{DW}(\secmatsym)$ and proves less precise for $0<\alpha<\pi/2$. $\Hr$ and $\Va$ generate vertical or horizontal boundaries based on either the gain or phase constraint. To introduce inclined boundaries, as depicted by the red lines in Fig.~\ref{DW_shell_union_numerical}, it is imperative to blend both the gain and phase constraints. The subsequent lemma facilitates a more accurate estimation of $\mathcal{DW}(\secmatsym)$, forming the basis for a pivotal outcome in this paper, connecting the norm and phase information of sectored-disk matrices.

\begin{lem}\label{lem:sectored_disk_LMI}
    Let $\gamma>0$, $\alpha\in[0,\pi/2)$ and $B\in\secmatsym$, then it holds
    \begin{equation*}
        B^*B \leq \frac{1}{2} \gamma\sec^2(\alpha)(B+B^*).
    \end{equation*}
\end{lem}
\begin{proof}
Since $0\leq \alpha<\pi/2$, $B\in\secmatsym$ must be sectorial or quasi-sectorial by definition.
Denote that $r:=\text{rank}(B)$.
Then there exist non-singular matrix $T$ and diagonal real matrix $\Lambda_r$ such that
\begin{equation}\label{origin_sectorial_decomposition}
    B = T^*\begin{bmatrix}I_r+j\Lambda_r \\ & 0\end{bmatrix}T.
\end{equation}
Partition the matrix $T$ as
\begin{equation*}
    T = 
    \begin{bmatrix}
        T_1 \\ T_2
    \end{bmatrix}.
\end{equation*}
with $T_1\in\mathbb{C}^{r\times n}$.

Then \eqref{origin_sectorial_decomposition} can be rewritten into $B = T_1^*(I_r+j\Lambda_r)T_1$,
and the Hermitian part of $B$ is given by $H(B)=T_1^*T_1$.
From the phase constraint on $B$, we have $\Lambda_r^2 \leq \tan^{2} (\alpha) I_r$.
Therefore, 
\begin{equation}\label{core_phase_bound}
    I_r+\Lambda_r^2 \leq \sec^2(\alpha)I_r.
\end{equation}

On the other hand, since $\bar{\sigma}(B) \leq \gamma$, we have 
$H(B)=T_1^*T_1\leq \gamma I_r$,
yielding that
\begin{equation}\label{T_estimation}
    T_1T_1^*\leq \gamma I_r.
\end{equation}

Combining \eqref{core_phase_bound} and \eqref{T_estimation}, we have
\begin{align*}
    B^*B 
    &= 
    T_1^*(I_r-j\Lambda_r)T_1
    T_1^*(I_r+j\Lambda_r)T_1 \\
    &\leq
    \gamma T_1^*(I_r-j\Lambda_r)(I_r+j\Lambda_r)T_1\\
    &= \gamma T_1^*(I_r+\Lambda_r^2)T_1\\
    &\leq \gamma \sec^2(\alpha)T_1^*T_1
    =\gamma\sec^2(\alpha) H(B),
\end{align*}
which completes the proof.
\end{proof}

Based on Lemma~\ref{lem:sectored_disk_LMI}, a tighter superset of $\mathcal{DW}(\secmatsym)$ is given by the following theorem.
\begin{thm}\label{thm_sectored_disk_DW_shell_inclusion}
    The DW shell union of $\secmatsym$ satisfies
    \begin{align*}
        \mathcal{DW}(\secmatsym)
        \subset \Hr \cap \Va \cap \Ko.
    \end{align*}
\end{thm}
\begin{proof}
    By Proposition \ref{prop_DW_shell_union_rough_shape}, it holds
    \begin{align*}
        \mathcal{DW}(\secmatsym)\subset
        \Hr  \cap \Va.
    \end{align*}
    It then suffices to show that
    \begin{equation*}
        \mathcal{DW}(\secmatsym) \subset\Ko.
    \end{equation*}
    Let $B\in\secmatsym$ and $(x,y,z)\in\mathcal{DW}(B)$, then
    there exists unit vector $u\in\mathbb{C}^n$ such that 
    \begin{align*}
        (x,y,z) = (\realpart (u^*Bu), \imagepart (u^*Bu), u^*B^*Bu).
    \end{align*}
    From Lemma \ref{lem:sectored_disk_LMI}, we obtain that
    \begin{align*}
        \gamma\sec^2(\alpha) x - z
        &=\gamma\sec^2(\alpha) \realpart (u^*Bu) - u^*B^*Bu \\
        &=u^*\left(\gamma\sec^2(\alpha) H(B) - B^*B\right)u \geq 0,
    \end{align*}
    which shows that
    \begin{align*}
        \mathcal{DW}(\secmatsym)\subset\left\{(x,y,z):~\gamma\sec^2(\alpha)x\geq z\right\}.
    \end{align*}
    The right-hand side of the above equation is exactly $\Ko$, which completes the proof.
\end{proof}

\begin{figure}[ht]
    \centering
    \vspace{-0.6cm}
    \includegraphics[scale = 0.45]{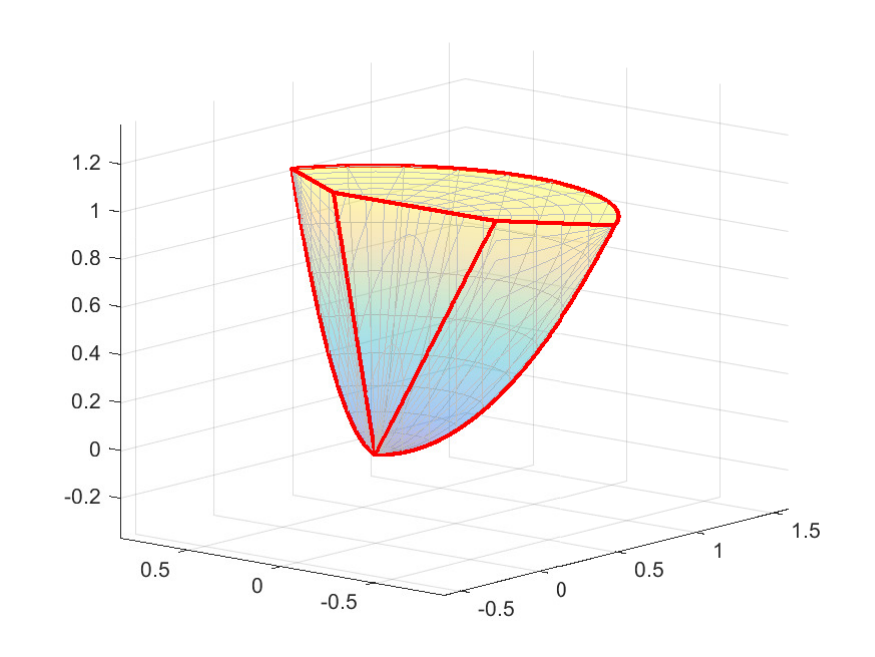}
    \includegraphics[scale = 0.45]{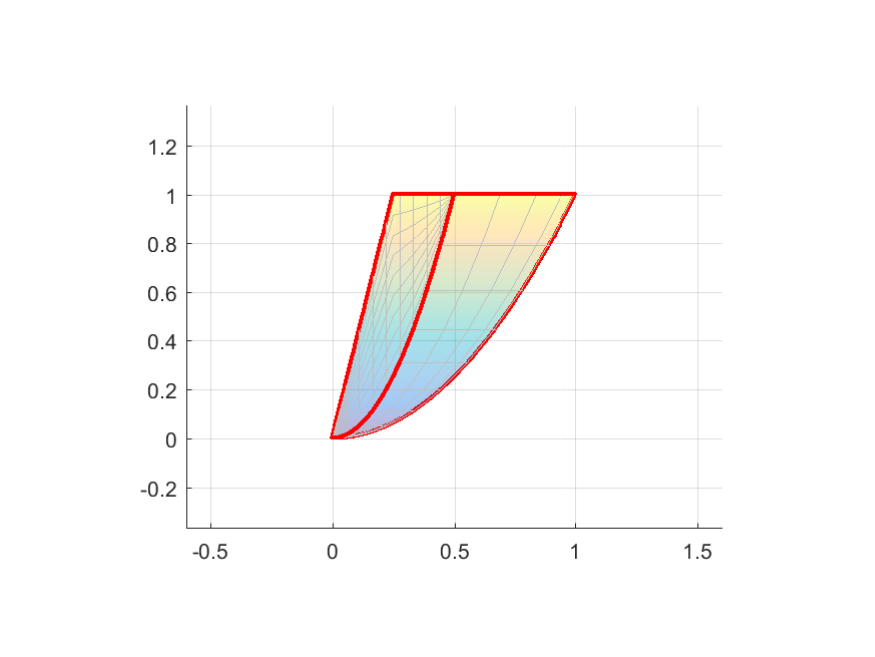}
    \vspace{-0.6cm}
    \caption{Two views of the superset of $\mathcal{DW}(\secmatsym)$ with $\alpha=\pi/3$ and $\gamma=1$.}\label{outer_bound1}
\end{figure}

\begin{figure}[ht]
    \vspace{-0.56cm}
    \centering
    \includegraphics[scale = 0.45]{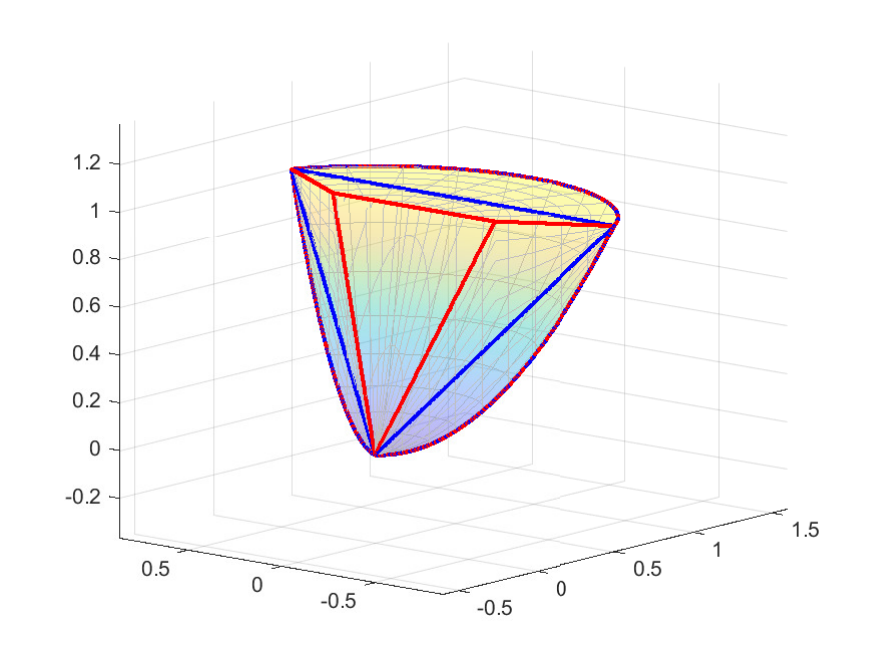}
    \caption{
        The subset and the superset of $\mathcal{DW}(\mathcal{S}_1(\pi/3))$.
        The blue lines sketch the subset and the red lines sketch the superset.}
    \label{DW_shell_superset_subset_comparasion}
\end{figure}

In Proposition~\ref{prop_DW_shell_union_rough_shape} and 
Theorem~\ref{thm_sectored_disk_DW_shell_inclusion},
a superset and a subset of $\mathcal{DW}(\secmatsym)$ are given by
$\mathcal{H}_\gamma \cap \mathcal{V}_\alpha \cap\mathcal{K}_{\gamma\sec(\alpha)}$ and
$\mathcal{H}_\gamma \cap \mathcal{V}_\alpha \cap\mathcal{K}_{\gamma\sec^2(\alpha)}$, respectively.
For $\alpha\in(0,\pi/2)$, the superset and the subset do not coincide.
Take $\alpha=\pi/3$ for example. The superset is visualized in Fig.~\ref{outer_bound1} and the gap between the superset and the subset
is visualized in Fig.~\ref{DW_shell_superset_subset_comparasion}.

\subsection{Necessary condition and sufficient condition for matrix sectored-disk problem}

In this subsection, we delve into both sufficient and necessary conditions for the matrix sectored-disk problem. These conditions align with the superset and subset of $\mathcal{DW}(\secmatsym)$ identified in the previous subsection, encapsulating the gap between the corresponding sufficient and necessary conditions.
It is worth noting that the conditions for a matrix sectored-disk problem are both necessary and sufficient when $B$ lies within the set of normal matrices or the set of half-disk matrices. In Proposition~\ref{prop_DW_shell_union_rough_shape} and Theorem~\ref{thm_sectored_disk_DW_shell_inclusion}, we have already characterized a superset and a subset of $\mathcal{DW}(\secmatsym)$, forming the basis for establishing a sufficient condition and a necessary condition for the sectored-disk problem.

Given the similar appearance of the superset from Proposition~\ref{prop_DW_shell_union_rough_shape} and the subset from Theorem~\ref{thm_sectored_disk_DW_shell_inclusion}, we present the following proposition to outline the separation condition for such sets.

\begin{prop}\label{thm_DW_shell_seapration_LMI}
    Let $\delta\geq\gamma>0, \alpha\in[0,\pi/2)$,
    and $A\in\Mn$ be an invertible matrix, then
    $\mathcal{DW}(-A^{-1})\cap(\Hr\cap\Va\cap\Kk)=\emptyset$ 
    if and only if there exist non-negative numbers $k_1,k_2,k_3,k_4$ such that
    \begin{multline}\label{DW_shell_seapration_LMI}
        k_1(I-\gamma^2 A^*A)
        +k_2 H(e^{-j(\pi/2-\alpha)}A) \\
        +k_3 H(e^{j(\pi/2-\alpha)}A)
        +k_4(\delta H(A)+I)
        > 0.
    \end{multline}
\end{prop}
\begin{proof}
    From Lemma~\ref{DW_shell_properties}, we have $\mathcal{DW}(-A^{-1})\subset\mathcal{P}$.
    Therefore, $\mathcal{DW}(-A^{-1})\cap(\Hr\cap\Va\cap\Kk)=\emptyset$ if and only if
    the intersection of $\mathcal{DW}(-A^{-1})$ and set
    \begin{multline*}
        \big\{(x,y,z)\in\mathbb{R}^3: z\leq\gamma^2,
        y\geq-\tan(\alpha)x,\\
        y\leq\tan(\alpha)x,
        z\leq \delta x\big\}
    \end{multline*}
    is empty.

    Notice that the right-hand side is the complementary set of a polyhedron.
    From the duality condition~\cite[Chapter 5]{Boyd2004ConvexOpt}, 
    this is equivalent to the existence of non-negative numbers $k_1,k_2,k_3,k_4$ such that
    \begin{multline*}
        \mathcal{DW}(-A^{-1})
        \subset
        \bigl\{k_1(\gamma^2-z)+k_2(\tan(\alpha)x+y)\\
        +k_3(\tan(\alpha)x-y)+k_4(\delta x-z)<0\bigl\}.
    \end{multline*}
    Thus, for all $(x,y,z)\in\mathcal{DW}(-A^{-1})$,
    \begin{multline*}
        k_1(\gamma^2-z)+k_2(\tan(\alpha)x+y)\\
        +k_3(\tan(\alpha)x-y)+k_4(kx-z)<0.
    \end{multline*}
    From the definition of DW shell, this is equivalent to that for each unit vector $u$,
    \begin{multline*}
        k_1\left(\gamma^2u^*u-u^*A^{-*}A^{-1}u\right)\\
        +k_2\left(\tan(\alpha) \realpart\left(u^*(-A^{-1})u\right) +\imagepart\left(u^*(-A^{-1})u\right)\right)\\
        +k_3\left(\tan(\alpha)\realpart\left(u^*(-A^{-1})u\right)-\imagepart\left(u^*(-A^{-1})u\right)\right)\\
        +k_4\left(\delta \realpart\left(u^*(-A^{-1})u\right)-u^*A^{-*}A^{-1}u\right)< 0.
    \end{multline*}
    Notice that $u$ can be relaxed to be any non-zero vector, 
    this inequality is equivalent to the following LMI
    \begin{multline*}
        k_1\left(\gamma^2 I-A^{-*}A^{-1}\right)
        +k_2\left(\tan(\alpha)H(-A^{-1})+K(-A^{-1})\right)\\
        +k_3\left(\tan(\alpha)H(-A^{-1})-K(-A^{-1})\right)\\
        +k_4\left(\delta H(-A^{-1})-A^{-*}A^{-1}\right)< 0.
    \end{multline*}
    This is further equivalent to 
    \begin{multline*}
        k_1(I-\gamma^2 A^*A)
        +k_2\sec(\alpha)(\sin(\alpha)H(A)-\cos(\alpha)K(A))\\
        +k_3\sec(\alpha)(\sin(\alpha)H(A)+\cos(\alpha)K(A))\\
        +k_4(\delta H(A)+I)>0,
    \end{multline*}
    which can be rewritten as
    \begin{multline*}
        k_1\left(I-\gamma^2 A^*A\right)
        +k_2\sec(\alpha)H(e^{-j(\pi/2-\alpha)}A)\\
        +k_3\sec(\alpha)H(e^{j(\pi/2-\alpha)}A)
        +k_4\left(\delta H(A)+I\right)> 0.
    \end{multline*}
    Without loss of generality, by replacing $k_3\sec(\alpha)$ and $k_4\sec(\alpha)$ with $k_3$ and $k_4$,
    we have \eqref{DW_shell_seapration_LMI}.
\end{proof}

By characterizing the subset of the DW shell union of sectored-disk matrices, 
a necessary condition for matrix sectored-disk problem can be derived, which is stated in the following theorem.

\begin{thm}\label{thm_matrix_sectored_disk_necessary}
    Let $\gamma>0, \alpha\in[0,\pi/2)$, and $A\in\Mn$ be an invertible matrix.
    If $\det(I+AB)\neq 0$ for all $B\in\secmatsym$, 
    then there exist non-negative numbers $k_1,k_2,k_3,k_4$ such that
    \begin{multline}\label{necessary_condition_LMI}
        k_1(I-\gamma^2 A^*A)
        +k_2 H(e^{-j(\pi/2-\alpha)}A) \\
        +k_3 H(e^{j(\pi/2-\alpha)}A)
        +k_4\left(\gamma~{\sec(\alpha)}H(A)+I\right)
        > 0.
    \end{multline}
\end{thm}
\begin{proof}
    From Corollary~\ref{corollary_DW_shell_separation}, that
    $\det(I+AB)\neq 0$ for all $B\in\secmatsym$ implies
    \begin{align*}
        \mathcal{DW}(-A^{-1})\cap\mathcal{DW}(\nsecmatsym)=\emptyset.
    \end{align*}
    From Proposition~\ref{prop_DW_shell_union_rough_shape}, 
    \begin{align*}
        \mathcal{DW}(\nsecmatsym)=\Hr\cap\Va\cap\Ki.
    \end{align*}
    
    Now we apply Proposition~\ref{thm_DW_shell_seapration_LMI} with 
    $\delta=\gamma~{\sec(\alpha)}$, then the LMI of \eqref{necessary_condition_LMI} must hold.
\end{proof}

When $A$ is invertible, a sufficient condition can be derived similarly in terms of the superset of the DW shell union of sectored-disk matrices.
Importantly, we find that this sufficient condition also holds for non-invertible matrix $A$, which is stated in what follows.
\begin{thm}\label{thm_matrix_sectored_disk_sufficient}
    Let $\gamma>0, \alpha\in[0,\pi/2)$, and $A\in\Mn$, then $\det(I+AB)\neq 0$ for all 
    $B\in\secmatsym$ if there exist non-negative numbers $k_1,k_2,k_3,k_4$ such that
    \begin{multline}\label{sufficient_condition_LMI}
        k_1(I-\gamma^2 A^*A)
        +k_2 H(e^{-j(\pi/2-\alpha)}A) \\
        +k_3 H(e^{j(\pi/2-\alpha)}A)
        +k_4 \left(\gamma\sec^2(\alpha)H(A)+I\right)
        > 0.
    \end{multline}
\end{thm}
\begin{proof}
    We will prove by contradiction. Suppose that $\det(I+AB)=0$, then there exists nonzero vector $v\in\mathbb{C}^n$ such that $(I+AB)v=0$.
    Let $u=Bv$ whereby $v=-Au$, then from \eqref{sufficient_condition_LMI}, we have
    \begin{multline*}
        k_1u^*(I-\gamma^2 A^*A)u
        +k_2u^* H(e^{-j(\pi/2-\alpha)}A)u \\
        +k_3u^* H(e^{j(\pi/2-\alpha)}A)u 
        +k_4u^*\left(\gamma\sec^2(\alpha)H(A)+I\right)u
        > 0.
    \end{multline*}
    With $u=Bv$ and $v=-Au$, above becomes
    \begin{multline*}
        k_1v^*(B^*B-\gamma^2 I)v 
        -k_2v^* H(e^{j(\pi/2-\alpha)}B)v \\
        -k_3v^* H(e^{-j(\pi/2-\alpha)}B)v \\
        +k_4v^*\left(B^*B-{\gamma}{\sec^2(\alpha)}H(B)\right)v 
        > 0.
    \end{multline*}

    On the other hand, since $B\in\secmatsym$, the first three terms of above inequality are negative semi-definite.
    It follows from Lemma~\ref{lem:sectored_disk_LMI} that the last term is also negative semi-definite.
    Since $k_i$ are non-negative numbers, summing up four terms weighted by $k_i$ leads to contradiction.
\end{proof}

Next, we extend the phase constraints to be asymmetric and present two corollaries derived from Theorem~\ref{thm_matrix_sectored_disk_necessary} and Theorem~\ref{thm_matrix_sectored_disk_sufficient}. The proofs are omitted as they are straightforward.
\begin{cor}\label{cor_general_necessary_condition}
    Given invertible $A\in\Mn$, $\gamma>0$ and $\alpha,\beta$ satisfying $-\pi\leq(\alpha+\beta)/2<\pi$ and
    $0<\beta-\alpha<\pi$,
    if $\det(I+AB)\neq 0$ for all $B\in\secmat$, 
    then there exist non-negative numbers $k_1,k_2,k_3,k_4$ such that
    \begin{multline}\label{general_necessary_condition_LMI}
        k_1\left(I-\gamma^2 A^*A\right)
        +k_2 H(e^{-j(\pi/2-\beta)}A) 
        +k_3 H(e^{j(\pi/2+\alpha)}A)\\
        +k_4 \left({\gamma}{\sec((\beta-\alpha)/2)}H(e^{j(\alpha+\beta)/2}A)+I\right)
        > 0.
    \end{multline}
\end{cor}

\begin{cor}\label{cor_general_sufficient_condition}
    Given $A\in\Mn$, $\gamma>0$ and $\alpha,\beta$ satisfying $-\pi\leq(\alpha+\beta)/2<\pi$ and
    $0<\beta-\alpha<\pi$, then $\det(I+AB)\neq 0$ for all 
    $B\in\secmat$ if there exist non-negative numbers $k_1,k_2,k_3,k_4$ such that
    \begin{multline}\label{general_sufficient_condition_LMI}
        k_1\left(I-\gamma^2 A^*A\right) 
        +k_3 H(e^{-j(\pi/2-\beta)}A) 
        +k_4 H(e^{j(\pi/2+\alpha)}A) \\
        +k_2\left({\gamma}{\sec^2((\beta-\alpha)/2)}H(e^{j(\alpha+\beta)/2}A)+I\right)
        > 0.
    \end{multline}
\end{cor}

The following is a simple illustration of the efficacy of the proposed conditions. 
\begin{expl}\label{sectored_disk_example}
    Consider matrix 
    \begin{align*}
        A = 
        \begin{bmatrix}
            0.58 - 0.21j & -0.92 + 0.41j &  0.35 - 0.90j\\
            0.91 + 0.31j &  0.69 - 0.93j &  0.51 - 0.80j\\
            0.31 - 0.65j &  0.86 - 0.44j &  0.48 + 0.64j
        \end{bmatrix}.
    \end{align*}
    First notice that $\bar{\sigma}(A)>1$ and $A$ is not sectorial, whereby neither the small gain nor the small phase analysis in terms of $B\in\mathcal{S}_1(\pi/3)$ is applicable here.
    Substitute matrix $A$ into LMI \eqref{sufficient_condition_LMI} with $\gamma=1$ and $\alpha=\pi/3$.
    Solve the coefficients by MOSEK solver~\cite{mosek}, which is interfaced by YALMIP~\cite{Lofberg2004} in MATLAB.
    The computational time is approximately 14ms.
    A feasible solution is given by $[k_1,k_2,k_3,k_4]=[1,5.2311,4.3742,0.0468]$.
    From Fig.~\ref{sectored_disk_example2}, one can see that the DW shell of $-A^{-1}$ is 
    separated from the superset of $\mathcal{DW}(\mathcal{S}_1(\pi/3))$.
    Thus, an application of Theorem~\ref{thm_matrix_sectored_disk_sufficient} yields that $\det(I+AB)\neq 0$ for all $B\in\mathcal{S}_1(\pi/3)$.
\end{expl}
\begin{figure}[ht]
    \centering
    \includegraphics[width = 0.8\columnwidth]{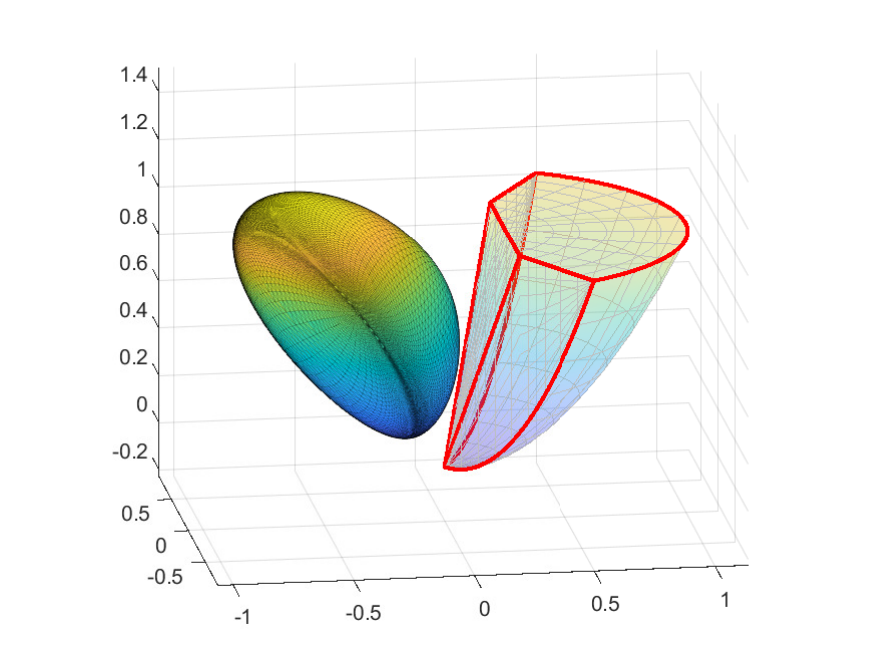}
    \caption{The left-hand side ellipse is $\mathcal{DW}(-A^{-1})$.
    The red lines on the right-hand side sketch a superset of $\mathcal{DW}(\mathcal{S}_1(\pi/3))$.}
    \label{sectored_disk_example2}
\end{figure}

\subsection{Solution to half-disk problem}
Recall that the shape of $\mathcal{DW}(\halfdiskmat)$ has been characterized in Corollary~\ref{cor_half_disk_union_shape}.
For the half-disk problem, the following proposition gives the necessary and sufficient condition 
by applying the separating condition to $\mathcal{DW}(\halfdiskmat)$.

\begin{prop}\label{prop_matrix_half_disk}
    Let $\gamma>0$ and $A\in\Mn$ be an invertible matrix, 
    then $\det(I+AB)\neq 0$ for all $B\in\halfdiskmat$ 
    if and only if there exists $k_1,k_2\geq 0$ such that
    \begin{align}\label{half_disk_A_LMI}
        k_1(I - \gamma^2 A^*A) + {k_2} H(A) > 0.
    \end{align}
\end{prop}
\begin{proof}
    From Corollary \ref{corollary_DW_shell_separation}, that
    $\det(I+AB)\neq 0$ for all $B\in\halfdiskmat$ is equivalent to
    \begin{align}\label{ninvM_halfdisk_separation}
        \mathcal{DW}(-A^{-1}) \cap \mathcal{DW}(\halfdiskmat) = \emptyset.
    \end{align}

    From Corollary~\ref{cor_half_disk_union_shape}, 
    $\mathcal{DW}(\halfdiskmat)=\mathcal{V}_{\pi/2} \cap \Hr$.
    Since $\mathcal{DW}(\halfdiskmat)$ and $\mathcal{DW}(-A^{-1})$ are both convex sets,
    their separation means the existence of a separating plane.
    Notice that $\mathcal{DW}(-A^{-1})\subset \mathcal{P}$, thus
    \begin{align*}
        \mathcal{DW}(-A^{-1}) \cap \mathcal{DW}(\halfdiskmat) = \emptyset
    \end{align*}
    is equivalent to 
    \begin{align*}
        \mathcal{DW}(-A^{-1}) \cap 
        \{(x,y,z)\in\mathbb{R}^3:~\gamma^2-z\geq 0,x \geq 0 \} = \emptyset.
    \end{align*}
    
    From the duality condition~\cite[Chapter 5]{Boyd2004ConvexOpt}, we have 
    \begin{multline*}
        {\{(x,y,z)\in\mathbb{R}^3:~\gamma^2-z\geq 0,x \geq 0 \}}^c= \\
        \{(x,y,z)\in\mathbb{R}^3:~\exists k_1,k_2\geq 0, k_1(\gamma^2-z) + k_2 x < 0\}.
    \end{multline*}
    It follows from the definition of DW shells that there exist $k_1,k_2>0$ such that for all unit vectors $u$,
    \begin{align*}
        k_1(\gamma^2 u^*u - u^*A^{-*}A^{-1}u) + {k_2} u^* H(-A^{-1})u < 0,
    \end{align*}
    which is equivalent to
    \begin{align*}
        k_1(A^{-*}A^{-1} - \gamma^2 I) + {k_2} H(A^{-1}) > 0.
    \end{align*}
    Multiplying the left-hand and right-hand sides with $A^*$ and $A$ respectively
    then yields \eqref{half_disk_A_LMI}, which completes the proof. 
\end{proof}

According to Corollary~\ref{corollary_DW_shell_separation}, matrix $I+AB$ is invertible if $\mathcal{DW}(-A^{-1})$ and $\mathcal{DW}(\secmatsym)$ are separated. Geometrically, Proposition~\ref{prop_matrix_half_disk} implies that the separating plane can be defined as a plane intersecting the point $(0,0,\gamma^2)$ with a normal vector of $(k_2,0,-k_1)$. A visual representation of this equation is illustrated in Fig.~\ref{half_disk_graphical_LMI}.

\begin{figure}[ht]
    \centering
    \begin{tikzpicture}[dot/.style={draw,fill,circle,inner sep=1pt},scale=1.5]
        \draw[dashed,domain=-cos(30):cos(30),variable=\t] plot ({\t},{1+\t*tan(30)}) node [right] {$z-\gamma^2=kx$};

        \draw (0,0) parabola (1,1) -- (0,1);
        \fill[pattern={Lines[angle=-45,distance={5pt}]}] (0,0) parabola (1,1) -- (0,1);
        
        \draw[->] (-1,0)--(1.5,0) node [right]{$x$};
        \draw[->] (0,-0.2)--(0,1.7) node [above left]{$z$};

    \end{tikzpicture}
    \caption{The shaded area is the XZ-plane projection of $\Hr\cap\Vhf$, 
    and the dashed line is determined by equation $k_1(\gamma^2-z)+k_2x=0$ with some $k_1,k_2$. }
    \label{half_disk_graphical_LMI}
\end{figure}
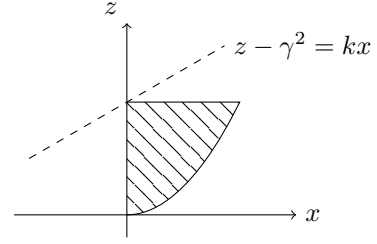

We would like to note that Proposition~\ref{prop_matrix_half_disk}, which is a special case of our main result in Theorem~\ref{thm_matrix_sectored_disk_sufficient}, provides results similar to those presented in \cite[Theorem 1]{Eszter1994RobustnessUC}. Specifically, \cite[Theorem 1]{Eszter1994RobustnessUC} calculates the SSV for half-disk uncertainty and establishes a necessary and sufficient condition for the half-disk problem using the mixed-$\mu$ technique. They also provide a geometric interpretation of their condition, which involves checking whether the generalized numerical range
\begin{multline*}
    \mathcal{W}(I-\gamma^2A^* A, A + A^*)= \\
    \{(u^*(I-\gamma^2 A^* A)u, u^*(A^*+A)u): u\in\mathbb{C}^n, u^*u=1\}
\end{multline*}
is separated from the closed negative orthant.
This condition is equivalent to \eqref{half_disk_A_LMI}.
Nevertheless, the development based on the DW-shell separation in our study may be regarded as an alternative way in exploiting and validating this fact. 

\subsection{Comparison with some existing results}

A commonly used approach for robust stability analysis in the presence of multiple quadratic constraints, such as the sectored-disk problem, is to employ the S-procedure~\cite{scherer2006LmiRelaxations, liu2017robust}. Given $N$ Hermitian matrices $\Pi_i \in \mathbb{C}^{n \times n}$, consider the set of matrices $B$ satisfying the following quadratic constraints simultaneously:
\begin{align*}
    \begin{bmatrix}
        B^* & I
    \end{bmatrix}
    \Pi_i
    \begin{bmatrix}
        B \\ I
    \end{bmatrix}
    \geq 0, i=1,...,N.
\end{align*}

Then by the S-procedure, we know that $\det(I+AB)\neq 0$ for all such $B$ if there exist 
$k_i\geq 0,i=1,...,N$ such that
\begin{align}\label{S_procedure_LMI}
    \sum\limits_{i=1}^N
    k_i 
    \begin{bmatrix}
        I & -A^*
    \end{bmatrix}
    \Pi_i
    \begin{bmatrix}
        I \\ -A
    \end{bmatrix}
    < 0, i=1,...,N.
\end{align}

As for the sectored-disk problem, $\Pi_1,\Pi_2,\Pi_3$ can be chosen according to the norm and phase constraints,
namely,
\begin{align*}
    \Pi_1 &= 
    \begin{bmatrix}
        -I & 0 \\ 0 & \gamma^2 I
    \end{bmatrix}, 
    \Pi_2 = 
    \begin{bmatrix}
        0 & e^{-j(\pi/2-\alpha)} I \\ e^{j(\pi/2-\alpha)} I & 0
    \end{bmatrix}, \\
    \Pi_3 &= 
    \begin{bmatrix}
        0 & e^{j(\pi/2-\alpha)} I \\ e^{-j(\pi/2-\alpha)} I & 0
    \end{bmatrix}.
\end{align*}
Substitute these $\Pi_i$ into \eqref{S_procedure_LMI}, 
we can derive a sufficient condition for the sectored-disk problem as that there exist $k_i \geq 0$ such that
\begin{multline}\label{sufficient_condition_by_S_procedure_LMI}
    k_1(I-\gamma^2 A^*A)
    +k_2 H(e^{-j(\pi/2-\alpha)}A) \\
    +k_3 H(e^{j(\pi/2-\alpha)}A)
    > 0.
\end{multline}

Using an analysis technique similar to Theorem~\ref{thm_DW_shell_seapration_LMI},
one can see that for invertible matrix $A$,
\eqref{sufficient_condition_by_S_procedure_LMI} is equivalent to
\begin{equation*}
    \mathcal{DW}(-A^{-1})\cap\Hr\cap\Va = \emptyset.
\end{equation*}
Recall  Example~\ref{sectored_disk_example}, and one can verify that LMI~\eqref{sufficient_condition_by_S_procedure_LMI} is infeasible for the matrix $A$ but \eqref{sufficient_condition_LMI} is feasible, 
demonstrating that the proposed condition of \eqref{sufficient_condition_LMI} is less conservative for the sectored-disk problem.

Tits et al.~\cite[Corollary 1]{Tits1999RobustnessUB} defined the phase sensitive structured singular value~(PS-SSV) of a matrix $A$ as
\begin{align*}
    \mu_\alpha(A):=(\inf
    \left\{\bar{\sigma}(B):B\in\mathcal{S}(\alpha), \det(I+AB)=0\right\})^{-1}.
\end{align*}
They proposed a sufficient condition for $\det(I+AB)\neq 0$ for all $B\in\mathcal{S}_1(\alpha)$, which is $$\mu_\alpha(A)<1.$$

Since the calculation of the PS-SSV is in general NP-hard, the authors provide an upper bound of the PS-SSV.
For the full block case, the upper bound is given by
\begin{multline*}
    \hat{\mu}_\alpha(A):=\big(\sup \{\gamma>0:~r>0, s>0, |b|\leq\cot(\alpha),\\
    r(\gamma^2 A^*A-I)-s((1+jb)A+(1-jb)A^*)<0\}\big)^{-1}.
\end{multline*}

On the other hand, with 
$$r = k_1, s = (k_2+k_3)\sin(\alpha), b = \dfrac{k_3-k_2}{k_3+k_2}\cot(\alpha),$$
\eqref{sufficient_condition_by_S_procedure_LMI} can be rewritten as
\begin{equation}\label{Tits1999_LMI}
    r (\gamma^2 A^*A - I) - s((1 + jb)A +(1 - jb)A^*) < 0.
\end{equation}
Thus, for invertible matrix $A$, it holds that 
\begin{align*}
    \hat{\mu}_\alpha(A)=
    (\sup\{\gamma>0:~\mathcal{DW}(-A^{-1})\cap\Hr\cap\Va=\emptyset\})^{-1}.
\end{align*}
In other words, for all $\gamma<\hat{\mu}^{-1}_\alpha(A)$, we have 
$\mathcal{DW}(-A^{-1})\cap\Hr\cap\Va=\emptyset$.

Comparing to condition \eqref{sufficient_condition_by_S_procedure_LMI} or 
\eqref{Tits1999_LMI} with Theorem~\ref{thm_matrix_sectored_disk_sufficient},
our result \eqref{sufficient_condition_LMI} has an additional term 
${\gamma}~{\sec^2(\alpha)}H(A)+I$.
Let
\begin{align*}
    \Pi_4 &=
    \begin{bmatrix}
        -I & \dfrac{\gamma}{2}\sec^2(\alpha) \\ \dfrac{\gamma}{2}\sec^2(\alpha) & 0
    \end{bmatrix}.
\end{align*}
We then have
\begin{align}
    \begin{bmatrix}
        I & -A^*
    \end{bmatrix}
    \Pi_4
    \begin{bmatrix}
        I \\ -A
    \end{bmatrix}
    =-{\gamma}{\sec^2(\alpha)}H(A)-I,\label{PI4_A_term}\\
    \begin{bmatrix}
        B^* & I
    \end{bmatrix}
    \Pi_4
    \begin{bmatrix}
        B \\ I
    \end{bmatrix}
    ={\gamma}{\sec^2(\alpha)}H(B)-B^*B.\label{PI4_B_term}
\end{align}
From Lemma~\ref{lem:sectored_disk_LMI}, \eqref{PI4_B_term} is positive definite for all $B\in\secmatsym$.

With Theorem~\ref{thm_matrix_sectored_disk_sufficient}, now we have a tighter upper bound of $\mu_\alpha(A)$, which is given by
\begin{multline*}
    \tilde{\mu}_\alpha(A):=\bigg(\sup \{\gamma>0:~r>0, s>0, t>0, |b|\leq\cot(\alpha),\\
    \quad\quad\quad\quad r(\gamma^2 A^*A-I)-s((1+jb)A+(1-jb)A^*) \\
    -t\Big({\gamma}{\sec^2\alpha}H(A)+I\Big)<0\}\bigg)^{-1}.
\end{multline*}

As $\hat{\mu}_\alpha(A)$ involves only the second-order term of $\gamma$, it can be computed using a quasi-convex optimization problem~\cite{Tits1999RobustnessUB}. However, due to the presence of both first and second-order terms, the calculation of $\tilde{\mu}_\alpha(A)$ becomes a nonlinear problem. Since $\tilde{\mu}_\alpha(A)$ is bounded within $[0, \hat{\mu}_\alpha(A)]$ and is monotonous with respect to $\gamma$, a line search can be employed as a feasible approach for its computation.

\begin{expl}
    Let $A$ be the matrix considered in Example~\ref{sectored_disk_example}.
    Consider the full block sectored-disk uncertainties with phase bound $[-\pi/3,\pi/3]$, we have 
    \begin{align*}
        \hat{\mu}_\alpha(A) = 1/0.5361 > 
        \tilde{\mu}_\alpha(A) = 1/1.4436.
    \end{align*}
\end{expl}
This example also shows that the sufficient condition in Theorem~\ref{thm_matrix_sectored_disk_sufficient} is strictly weaker (and hence less conservative) than the sufficient condition \eqref{sufficient_condition_by_S_procedure_LMI}.
In general, we have $\mu_\alpha(A)<\hat{\mu}_\alpha(A)$ for $\alpha\in[0,\pi/2)$, and 
$\mu_\alpha(A)=\hat{\mu}_\alpha(A)$ only holds for $\alpha=\pi/2$.
This gives an answer to an open problem in~\cite[Problem 42]{OpenProblems1999}.

\section{Sectored-disk uncertainties}\label{Sectored-disk uncertainties}
In the previous section, we established both sufficient and necessary conditions to address the matrix sectored-disk problem. In this section, we extend the sectored-disk problem along with the proposed conditions to the scenario involving LTI systems. First, we formulate a sectored-disk problem for LTI systems, with the aim of determining conditions on the given LTI system $G$ so that the feedback system $G\#\Delta$ is robustly stable against each uncertainty $\Delta$ that belongs to a prescribed sectored-disk uncertainty set frequency-wise. 
Next, we develop a robust feedback stability condition based on the matrix result for resolving the problem.
Finally, leveraging the state-space representations of LTI systems and employing the Kalman-Yakubovich-Popov (KYP) lemma, we transform the robust feedback stability condition into solving LMIs for ease of computation and verification.

\subsection{Robust stability with sectored-disk uncertainty}
Given an uncertain system, a common modelling method is to sample its frequency 
response in a frequency-wise manner so that the norm and phase of the uncertain system can be estimated.
In this case, the uncertain feedback system can be modelled as $G \# \Delta$ in Fig.~\ref{fig_uncertain_system}.
Let the parameters describing the uncertain system satisfy
$\alpha(\omega), \beta(\omega)\in(-\pi,\pi]$ with $(\beta(\omega)+\alpha(\omega))/2\in(-\pi,\pi]$,
$\beta(\omega)-\alpha(\omega)<\pi$ and $\gamma(\omega)>0$.
We will focus on the following set of sectored-disk uncertainty
\begin{multline*}
    \secsys := 
    \{ \Delta\in\RHn :~ \| \Delta(j\omega) \| \leq \gamma(\omega), \\
    \Delta \text{ is frequency-wise semi-sectorial}, \\
    [\underline{\phi}(\Delta(j\omega)),\bar{\phi}(\Delta(j\omega))]\subset[\alpha(\omega),\beta(\omega)],
    \forall \omega\in[0,\infty].
    \}
\end{multline*}
Notice that $\Delta(-j\omega)=\overline{\Delta(j\omega)}$ for $\Delta(s)\in\secsys$ by the conjugate symmetric property of real-rational transfer matrices.
We have $\Phi_\infty(\Delta(j\omega))\subset[-\beta(-\omega),-\alpha(-\omega)]$ for all $\omega\in[-\infty,0)$.
For notational simplicity, when the phase constraints are symmetric, we also denote 
\begin{align*}
    \secsyssym := \mathcal{U}_\gamma(-\alpha,\alpha).
\end{align*}

Taking advantage of the results of the matrix sectored-disk problem,
we obtain the following robust feedback stability condition for the sectored-disk problem.
\begin{thm}\label{thm_system_general_sectored_disk}
    Let $G\in\RHn$, $\gamma(\omega)>0$ and $\alpha(\omega), \beta(\omega)\in(-\pi,\pi]$ with 
    $\beta(\omega)-\alpha(\omega)<\pi$, for all $\omega\in[0,\infty]$.
    Then $G \# \Delta$ is stable for all $\Delta\in\mathcal{U}_\gamma(\alpha,\beta)$ 
    if there exist scalar non-negative functions $k_1(\omega),k_2(\omega),k_3(\omega),k_4(\omega)$ such that for all $\omega\in[0,\infty]$,
    \begin{equation}\label{sys_sufficient_condition}
        \sum\limits_{i=1}^4 k_i(\omega) T_i(\omega) > 0,
    \end{equation}
    where
    \begin{align*}
        T_1(\omega) &= I-\gamma(\omega)^2 G^*(j\omega)G(j\omega), \\
        T_2(\omega) &= 2H(e^{-j(\pi/2-p(\omega)-q(\omega))}G(j\omega)), \\
        T_3(\omega) &= 2H(e^{j(\pi/2+q(\omega)-p(\omega))}G(j\omega)),\\
        T_4(\omega) &= {\gamma(\omega)}{\sec^2(p(\omega)}H(e^{jq(\omega)}G(j\omega))+I,\\
        p(\omega) &= (\beta(\omega)-\alpha(\omega))/2,
        q(\omega) = (\beta(\omega)+\alpha(\omega))/2.
    \end{align*}
\end{thm}
\begin{proof}
    Note that $G,\Delta\in\RHn$. In order to prove that the feedback system $G\#\Delta$ is stable, it suffices to show that $\det(I+G(j\omega)\Delta(j\omega))\neq 0$ for all $\omega\in[-\infty,\infty]$ and that $(I+G(s)\Delta(s))^{-1}$ has no pole on the open right half plane. 

    First, note that inequality \eqref{sys_sufficient_condition} holds for $G$ implies that it holds for any $cG$ with $c\in(0,1]$.
    It then follows from Corollary~\ref{cor_general_sufficient_condition} that
    $\det(I+cG(j\omega)\Delta(j\omega))\neq 0$ for all $\omega\in[0,\infty]$ and $c\in(0,1]$.
    Then by the conjugate symmetry of $G(j\omega)$, we obtain that
    $\det(I+cG(j\omega)\Delta(j\omega))\neq 0$ for all $\omega\in[-\infty,\infty]$ and $c\in(0,1]$.

    Second, suppose to the contrapositive that $(I+G(s)\Delta(s))^{-1}$ has some poles on the open right half plane. 
    Note that the poles of $(I+cG(s)\Delta(s))^{-1}$ on the open right half plane vary continuously in terms of parameter $c\in(0,1]$ and when $c_0 > 0$ is sufficiently small, $(I+c_0G(s)\Delta(s))^{-1}$ has no pole on the open right half plane by the small gain condition. Then by the assumption that $(I+G(s)\Delta(s))^{-1}$ has some poles on the open right half plane, there must be some $c\in(c_0,1)$ such that some of the poles of $(I+cG(s)\Delta(s))^{-1}$ are on the imaginary axis. This contradicts to that $\det(I+cG(j\omega)\Delta(j\omega))\neq 0$ for all $\omega\in[-\infty,\infty]$ and $c\in(0,1]$. By contraposition, we conclude that $(I+G(s)\Delta(s))^{-1}$ has no pole on the open right half plane, which completes the proof. \hfill \qedhere
\end{proof}

The following is an example on showing how we solve the frequency-dependent inequality \eqref{sys_sufficient_condition} to obtain the robust feedback stability. 
\begin{expl}
    Consider a plant $G(s)$ which is given by
    \begin{equation*}
        G(s) = 
        \begin{bmatrix}
            \frac{\strut 2s+1}{\strut 5s+1} & \frac{\strut 12s}{\strut 10s+1} \\
            \frac{\strut s}{\strut 20s+1} & \frac{\strut 5s+2}{\strut 8s+1}
        \end{bmatrix}
    \end{equation*}
    Clearly, we see that $G\in\mathcal{RH}^{2\times 2}_\infty$ and $G^{-1} \in\mathcal{L}^{2\times 2}_\infty$.
    Consider that the sectored-disk uncertainty is described by bounds $\gamma(\omega), \alpha(\omega), \beta(\omega)$, where
    \begin{align*}
        \gamma(\omega) & = 
        \left\{
        \begin{array}{cc}
            10, & 0\leq \omega \leq \pi/9, \\
            4, & \omega \geq \pi/9.
        \end{array}
        \right. \\
        \beta(\omega) & = -\alpha(\omega) = 
        \left\{
        \begin{array}{cc}
            \pi/2, & 0 \leq \omega \leq \pi/9, \\
            \pi/3, & \pi/9 \leq \omega \leq \pi/6, \\
            \pi/4, & \pi/6 \leq \omega \leq \pi/3, \\
            \pi/6, & \omega \geq \pi/3. \\
        \end{array}
        \right.
    \end{align*}

    \vspace{-0.5cm}
    \begin{figure}[ht]
        \centering
        \includegraphics[width = \columnwidth]{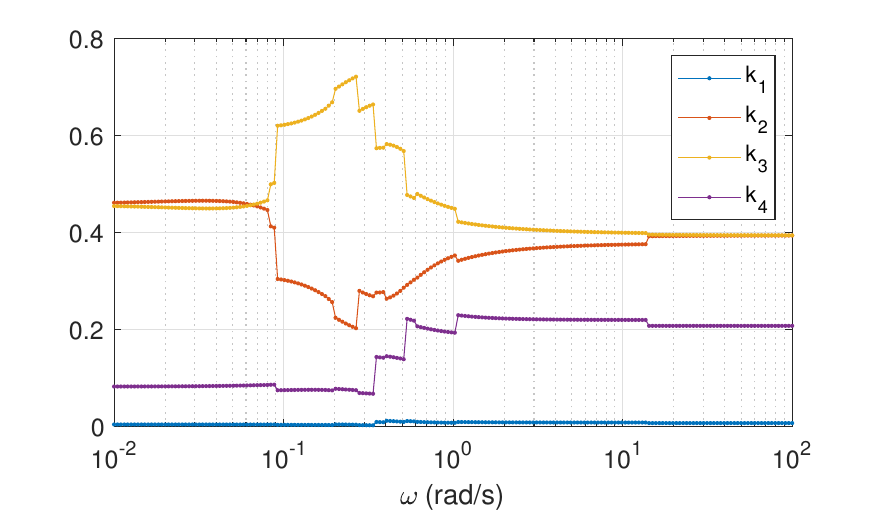}
        \caption{The values of $k_i(\omega),~i=1,2,3,4$.}
        \label{fig:system_condition_ki}
    \end{figure}
    
    In this example, we evenly sample the frequency interval $[10^{-2},10^2]$ rad/s by 200 grid points.
    The total computation time is about 36.96s.
    Then solving $k_i(\omega)>0$ in \eqref{sys_sufficient_condition}, we obtain the values of $k_i(\omega)$, as shown in Fig.~\ref{fig:system_condition_ki}.
    Therefore, Theorem~\ref{thm_system_general_sectored_disk} holds with such $k_i(\omega)$, and $G \# \Delta$ is stable for all $\Delta\in\mathcal{U}_\gamma(\alpha,\beta)$. 
\end{expl}

\subsection{State-space robust stability conditions}
Given a system $G\in\RHn$ with a minimum realization 
$\left[\begin{array}{c|c}A&B\\ \hline C&D\end{array}\right]$,
the KYP lemma~\cite{liu2016robust} connects the time-domain system characterization 
and the frequency-domain description as well as the stability condition.
When the gain bound and phase bound of sectored-disk uncertainty are constants with respect to the frequency, the following theorem utilizes the KYP lemma to provide a numerical method for verifying the closed-loop robust stability condition.

\begin{thm}\label{thm_sectored_disk_lemma}[Sectored-disk lemma]
    Let $G (s)=C (sI-A)^{-1}B+D\in\RHn$, $\gamma >0$, and $\alpha\in[0,\pi/2]$. 
    Then $G \# \Delta$ is stable for all
    $\Delta\in\mathcal{U}_\gamma(\alpha)$ if there exist $P>0$ and $k_i \geq 0,~i=1,2,3,4$, such that
    \begin{align}\label{sectored_disk_lemma_LMI}
        \begin{bmatrix}
            A^* P + PA & PB \\
            B^* P & 0
        \end{bmatrix}
        +\sum\limits_{i=1}^4 k_iM_i
        < 0,
    \end{align}
    where
    \begin{align*}
        &M_1 = 
        \begin{bmatrix}
            \gamma^2 C^*C & \gamma^2 C^*D \\
            \gamma^2 D^*C & \gamma^2 D^*D - I
        \end{bmatrix}, \\
        &M_2 = 
        \begin{bmatrix}
            0 & -e^{-j(\pi/2-\alpha)} C^* \\
            -e^{j(\pi/2-\alpha)}C & - 2 H \left(e^{j(\pi/2-\alpha)} D\right)
        \end{bmatrix}, \\
        &M_3 = 
        \begin{bmatrix}
            0 & -e^{j(\pi/2-\alpha)} C^* \\
            -e^{-j(\pi/2-\alpha)}C & - 2 H \left(e^{-j(\pi/2-\alpha)}D\right)
        \end{bmatrix}, \\
        &M_4 = 
        \begin{bmatrix}
            0 & -C^* \\
            -C & -(D^*+ D) - (2\cos^2(\alpha)/\gamma)I
        \end{bmatrix}.
    \end{align*}
\end{thm}

\begin{proof}
   Note that there exist $P>0$ and $k_i\geq 0,i=1,2,3,4$ such that \eqref{sectored_disk_lemma_LMI} holds. Then by KYP lemma, \eqref{sectored_disk_lemma_LMI} is equivalent to for $\omega\in[-\infty,\infty]$,
    \begin{align}\label{sectored_disk_lemma_LMI_KYP_form1}
        \sum\limits_{i=1}^4 k_i
        \begin{bmatrix}
            (j\omega I-A)^{-1}B \\ I
        \end{bmatrix}^*
        M_i
        \begin{bmatrix}
            (j\omega I-A)^{-1}B \\ I
        \end{bmatrix}<0.
    \end{align}

    On the other hand, noting $G(j\omega) = C(j\omega I-A)^{-1}B+D$, we have the following sequence of equations:
    \begin{multline*}
        \gamma^2 G^*(j\omega) G(j\omega) - I\\ = 
        \begin{bmatrix}
            (j\omega I-A)^{-1}B \\ I
        \end{bmatrix}^*
        M_1
        \begin{bmatrix}
            (j\omega I-A)^{-1}B \\ I
        \end{bmatrix},
    \end{multline*}
    \begin{multline*}
        -2H(e^{j(\pi/2-\alpha)}G(j\omega))\\ = 
        \begin{bmatrix}
            (j\omega I-A)^{-1}B \\ I
        \end{bmatrix}^*
        M_2
        \begin{bmatrix}
            (j\omega I-A)^{-1}B \\ I
        \end{bmatrix},
    \end{multline*}
    \begin{multline*}
        -2H(e^{-j(\pi/2-\alpha)}G(j\omega))\\ = 
        \begin{bmatrix}
            (j\omega I-A)^{-1}B \\ I
        \end{bmatrix}^*
        M_3
        \begin{bmatrix}
            (j\omega I-A)^{-1}B \\ I
        \end{bmatrix}.
    \end{multline*}
    \begin{multline*}
        -2H(G(j\omega))-(2\cos^2(\alpha)/\gamma)I\\ = 
        \begin{bmatrix}
            (j\omega I-A)^{-1}B \\ I
        \end{bmatrix}^*
        M_4
        \begin{bmatrix}
            (j\omega I-A)^{-1}B \\ I
        \end{bmatrix},
    \end{multline*}
    Therefore, \eqref{sectored_disk_lemma_LMI_KYP_form1} is further equivalent to
    \begin{multline*}
        k_1(I-\gamma^2 G^*(j\omega) G(j\omega))
        +2 k_2 H(e^{j(\pi/2-\alpha)}G(j\omega))\\
        +2 k_3 H(e^{-j(\pi/2-\alpha)}G(j\omega))\\
        +k_4(G(j\omega)+G^*(j\omega)+(2\cos^2(\alpha)/\gamma)I) > 0.
    \end{multline*}
    From Theorem~\ref{thm_system_general_sectored_disk}, 
    this implies for all $\Delta\in\mathcal{U}_\gamma(\alpha)$, $G \# \Delta$ is stable.
    This completes the proof.
\end{proof}

\begin{expl}
    Consider the following SISO system
    \begin{align*}
        G=\left[
        \begin{array}{c|c}
            A & B \\ \hline C & D
        \end{array}
        \right]=
        \left[
        \begin{array}{cc|c}
            0.3442 & 1.1386 & 1.6975 \\
            -1.0904 & -0.8495 & -0.8061 \\ 
            \hline 
            0.5363 & 0.3336 & -0.2373
        \end{array}
        \right].
    \end{align*}
    Solving inequality \eqref{sectored_disk_lemma_LMI} with $\gamma = 1$ and $\alpha = \pi/3$, we have  
    \begin{align*}
        P&=
        \begin{bmatrix}
            15.5281 & 9.0543 \\ 9.0543 & 15.9003
        \end{bmatrix},\\
        k_1 = 28.1602,
        k_2 &= 6.4926,
        k_3 = 6.4926,
        k_4 = 11.8703.
    \end{align*}
    Therefore, $G \# \Delta$ is stable for all sectored-disk uncertainties  $\Delta\in\mathcal{U}_\gamma(\alpha)$.
    Fig.~\ref{fig:Gs_ss_example} is the Nyquist plot of $G$. 
    From the Nyquist plot, one can see that $G$ is not passive.
    Moreover by calculating the norm of $G$, one has $\|G\|_\infty = 1.1568 > 1$.
    Therefore, the closed-loop robust stability cannot be ensured by the
    small gain theorem nor the small phase theorem.
    \begin{figure}[ht]
        \centering
        \includegraphics[scale = 0.5]{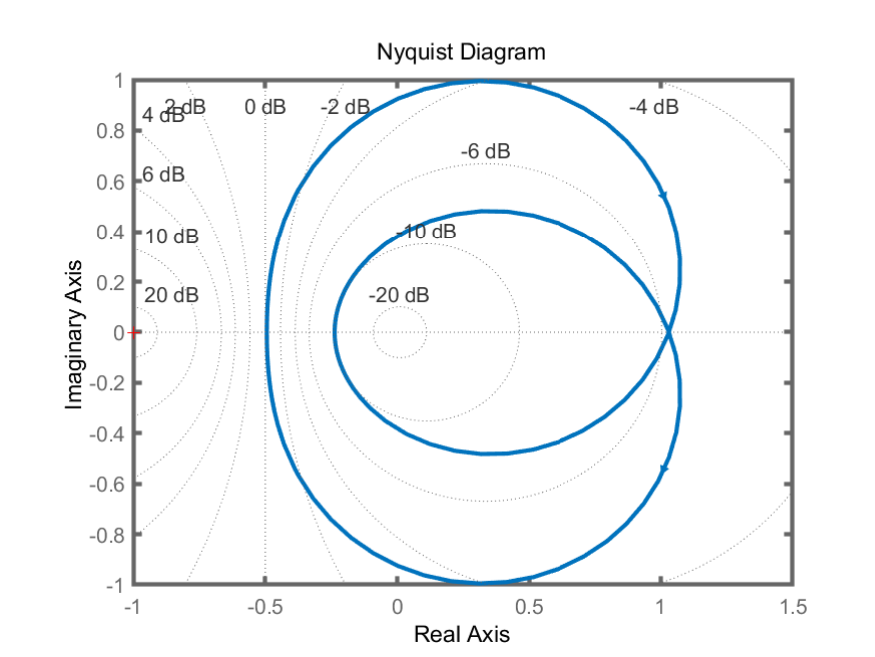}
        \caption{Nyquist plot of $G$ such that
        $G\#\Delta$ is stable for all $\Delta\in\secsyssym$ with $\alpha=\pi/3$.}
        \label{fig:Gs_ss_example}
    \end{figure}
\end{expl}

\begin{expl}
    Consider another MIMO system 
    \begin{align*}
        G=&\left[
        \begin{array}{c|c}
            A & B \\ \hline C & D
        \end{array}
        \right] \\
        =&
        \left[
            \begin{array}{ccc|cc}
                -0.699 & 0.044 & 0.855 & 0.812 & 0.044 \\ 
                0.418 & -0.477 & -0.568 & 0.361 & -0.792 \\ 
                -0.639 & -0.074 & -0.998 & 0.029 & 0.998 \\ 
                \hline 
                0.359 & 0.393 & 0.543 & -0.248 & -0.847 \\ 
                0.625 & 0.077 & 0.591 & -0.044 & -0.048 \\ 
            \end{array}
        \right].
    \end{align*}
    Solving Theorem~\ref{thm_sectored_disk_lemma} with $\gamma = 1$ and $\alpha = \pi/3$, we have 
    \begin{align*}
        P&=
        \begin{bmatrix}
            14.5345 & 9.5677 & 10.0333 \\
            9.5677 & 11.9948 & 10.2902 \\
            10.0333 & 10.2902 & 18.0862
        \end{bmatrix},\\
        k_1 &= 20.8005,
        k_2 = 0.8003,
        k_3 = 0.8003,
        k_4 = 5.4204.
    \end{align*}
    Therefore, $G \# \Delta$ is stable for all sectored-disk uncertainties $\Delta\in\mathcal{U}_\gamma(\alpha)$.
    Applying the bounded real lemma, we can check that $G$ is not norm-bounded by 1.
    The sector real lemma is also infeasible in this case.
    Therefore, the closed-loop robust stability cannot be ensured by the
    small gain theorem nor the small phase theorem.
\end{expl}

\begin{rem}
    Since the $k_i$, $i=1,2,3,4$, in Theorem~\ref{thm_sectored_disk_lemma} are constants, it gives a more conservative result than that in Theorem~\ref{thm_system_general_sectored_disk}.
    The following example would show the difference.
    Let a system $G$ be given by
    \begin{align*}
        G=&\left[
        \begin{array}{c|c}
            A & B \\ \hline C & D
        \end{array}
        \right] \\
        =&
        \left[
            \begin{array}{ccc|cc}
                -3.803 & -1.134 & 3.474 & 5.036 & 3.568 \\
                4.573 & -12.656 & 5.861 & 10.204 & 10.512 \\
                1.559 & 7.793 & -15.323 & 10.939 & 13.879 \\ 
                \hline 
                0.075 & 0.131 & 0.165 & -0.218 & -0.278 \\ 
                0.378 & 0.058 & 0.128 & -0.641 & -0.575 \\ 
            \end{array}
        \right].
    \end{align*}
    
    Solving Theorem~\ref{thm_sectored_disk_lemma} with $\gamma=1$ and $\alpha = \pi/3$, the problem is infeasible, revealing that Theorem~\ref{thm_sectored_disk_lemma} is inapplicable here. 

    On the other hand, we divide the frequency interval $[10^{-2},10^2]$ rad/s into 200 grid points, then check the separating condition in Theorem~\ref{thm_system_general_sectored_disk} frequency-wise. 
    The total computation time is about 40.69s by MOSEK solver interfaced by YALMIP in MATLAB.
    One can see that there exist functions $k_i(\omega), i=1,2,3,4$, as in Fig.~\ref{fig:sectored_real_lemma_example_ki}, whereby Theorem~\ref{thm_system_general_sectored_disk}  is applicable.

    \begin{figure}[ht]
        \centering
        \includegraphics[width=\columnwidth]{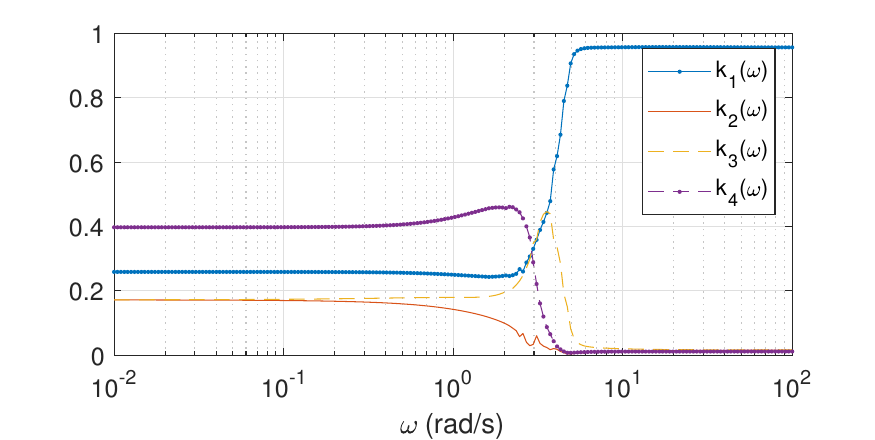}
        \caption{Functions $k_i$.}
        \label{fig:sectored_real_lemma_example_ki}
    \end{figure}
\end{rem}

In Theorem~\ref{thm_sectored_disk_lemma}, the KYP lemma is applied to derive an LMI condition
for robust stability under the symmetric phase bounded uncertainty. In what follows, we 
consider an uncertainty $\Delta\in\secsys$ that has asymmetric phase constraints, then for all $\omega\in[0,\infty]$,
$\Phi_\infty(\Delta(j\omega))\subset[\alpha,\beta]$ and $\bar{\sigma}(\Delta(j\omega)) \leq \gamma$.
Since the $\omega$ only takes non-negative values in the constraints, the classical KYP lemma does not apply.
To tackle this problem, the generalized KYP lemma would help.

Define a curve in the complex plane via
\begin{align*}
\Lambda(\Sigma,\Psi) = \left\{\lambda\in\mathbb{C}~\bigg|~\begin{bmatrix}
	\lambda \\ 1
\end{bmatrix}^*\Sigma\begin{bmatrix}
\lambda \\ 1
\end{bmatrix} = 0,~\begin{bmatrix}
\lambda \\ 1
\end{bmatrix}^*\Psi\begin{bmatrix}
\lambda \\ 1
\end{bmatrix}\geq 0
\right\}\end{align*}
with parameters $\Sigma$ and $\Psi$ being Hermitian matrices of proper dimensions. 
\begin{lem}\label{lem:gem_KYP}[generalized KYP lemma,\cite{Iwasaki2005GeneralizedKL}]
    Let $A\in\Mn, B\in\mathbb{C}^{n\times m}$,
    $M=M^*\in\mathbb{C}^{(n+m)\times(n+m)}$.
    Suppose $\left[\begin{array}{c|c}A&B\end{array}\right]$ is controllable.
    Let $\Omega$ be the set of eigenvalues of $A$ in $\Lambda(\Sigma,\Psi)$.
    Then
    \begin{align*}
        \begin{bmatrix}
            (\lambda I-A)^{-1}B \\ I
        \end{bmatrix}^*
        M
        \begin{bmatrix}
            (\lambda I-A)^{-1}B \\ I
        \end{bmatrix}
        <0
    \end{align*}
    for all $\lambda\in\Lambda(\Sigma,\Psi) \backslash \Omega$
    if and only if there exist Hermitian $X$ and $Y$ such that 
    \begin{align*}
        Y > 0,
        \begin{bmatrix}
            A&B\\I&0
        \end{bmatrix}^*
        (\Sigma\otimes X+\Psi\otimes Y)
        \begin{bmatrix}
            A&B\\I&0
        \end{bmatrix}
        +M<0.
    \end{align*}
\end{lem}

With the generalized KYP lemma, the following theorem then provides an LMI condition for robust stability
under uncertainty set $\secsys$.
Denote that
\begin{align*}
    p: = {\gamma}{\sec^2(\beta-\alpha)},~~
    q: = \dfrac{1}{2}(\alpha+\beta).
\end{align*}
\begin{thm}\label{generalized_sectored_disk_lemma}
    Let $\gamma>0$, $-\pi\leq(\alpha+\beta)/2<\pi$, $0<\beta-\alpha<\pi$, $G\in\RHn$ with a minimal realization $\left[\begin{array}{c|c}A & B \\ \hline C & D\end{array}\right]$.
    Then $G \# \Delta$ is stable for all $\Delta\in\secsys$ if there exist Hermitian $X$ and $Y$, and scalars $k_i \geq 0,~i=1,...,4$, such that $Y>0$ and
    \begin{align}
        \begin{bmatrix}\label{generalized_sectored_disk_lemma_condition}
            A & B \\ I & 0
        \end{bmatrix}^*
        \begin{bmatrix}
            0 & X+jY \\ X-jY & 0
        \end{bmatrix}
        \begin{bmatrix}
            A & B \\ I & 0
        \end{bmatrix}
        +\sum\limits_{i=1}^4 k_iM_i<0,
    \end{align}
    where
    \begin{align*}
        &M_1 = 
        \begin{bmatrix}
            \gamma^2 C^*C & \gamma^2 C^*D \\
            \gamma^2 D^*C & \gamma^2 D^*D - I
        \end{bmatrix}, \\
        &M_2 = 
        \begin{bmatrix}
            0 & -e^{j(\pi/2-\beta)} C^* \\
            -e^{-j(\pi/2-\beta)}C & - 2 H(e^{-j(\pi/2-\beta)}D)
        \end{bmatrix}, \\
        &M_3 = 
        \begin{bmatrix}
            0 & -e^{-j(\pi/2-\alpha)} C^* \\
            -e^{j(\pi/2-\alpha)}C & - 2 H(e^{j(\pi/2-\alpha)}D)
        \end{bmatrix}, \\
        &M_4 = 
        \begin{bmatrix}
            0 & -pe^{-jq} C^* \\
            -pe^{jq}C & - 2p H(e^{jq}D) + I
        \end{bmatrix}.
    \end{align*}
\end{thm}
\begin{proof}
    It follows by the premise and Lemma~\ref{lem:gem_KYP} that \eqref{generalized_sectored_disk_lemma_condition} is equivalent to 
 that for $\omega\in[0,\infty]$,
    \begin{align}\label{generalized_sectored_disk_lemma_LMI_KYP_form1}
        \sum\limits_{i=1}^4 k_i
        \begin{bmatrix}
            (j\omega I-A)^{-1}B \\ I
        \end{bmatrix}^*
        M_i
        \begin{bmatrix}
            (j\omega I-A)^{-1}B \\ I
        \end{bmatrix}<0.
    \end{align}

    By $G(j\omega) = C(j\omega I-A)^{-1}B+D$,  we have
    \begin{multline*}
        \gamma^2 G^*(j\omega) G(j\omega) - I\\ = 
        \begin{bmatrix}
            (j\omega I-A)^{-1}B \\ I
        \end{bmatrix}^*
        M_1
        \begin{bmatrix}
            (j\omega I-A)^{-1}B \\ I
        \end{bmatrix},
    \end{multline*}
    \begin{multline*}
        -2H(e^{-j(\pi/2-\beta)}G(j\omega))\\ = 
        \begin{bmatrix}
            (j\omega I-A)^{-1}B \\ I
        \end{bmatrix}^*
        M_2
        \begin{bmatrix}
            (j\omega I-A)^{-1}B \\ I
        \end{bmatrix},
    \end{multline*}
    \begin{multline*}
        -2H(e^{j(\pi/2+\alpha)}G(j\omega))\\ = 
        \begin{bmatrix}
            (j\omega I-A)^{-1}B \\ I
        \end{bmatrix}^*
        M_3
        \begin{bmatrix}
            (j\omega I-A)^{-1}B \\ I
        \end{bmatrix}.
    \end{multline*}
    \begin{multline*}
        -2pH(e^{jq}G(j\omega))-I\\ = 
        \begin{bmatrix}
            (j\omega I-A)^{-1}B \\ I
        \end{bmatrix}^*
        M_4
        \begin{bmatrix}
            (j\omega I-A)^{-1}B \\ I
        \end{bmatrix},
    \end{multline*}
    Therefore, \eqref{generalized_sectored_disk_lemma_LMI_KYP_form1} is further equivalent to
    \begin{multline*}
        k_1(I-\gamma^2 G^*(j\omega) G(j\omega))
        +k_2H(e^{-j(\pi/2-\beta)}G(j\omega))\\
        +k_3H(e^{j(\pi/2+\alpha)}G(j\omega))
        +k_4(pH(G(j\omega))+I) > 0.
    \end{multline*}
   This implies by  Theorem~\ref{sectored_disk_lemma_LMI} that for all $\Delta\in\mathcal{U}_\gamma(\alpha,\beta)$, $G \# \Delta$ is stable, which completes the proof.
\end{proof}

\begin{expl}
    Consider another MIMO system 
    \begin{align*}
        G=&\left[
        \begin{array}{c|c}
            A & B \\ \hline C & D
        \end{array}
        \right] \\
        =&
        \left[
            \begin{small}
                \begin{array}{ccc|cc}
                    -1.908 & -0.894 & 1.635 & 0.349 & 1.517 \\ 
                    1.846 & 1.882 & 1.441 & -1.796 & 0.306 \\ 
                    1.735 & 1.847 & -1.601 & 1.983 & -1.809 \\ 
                    \hline 
                    0.537 & 0.356 & 0.666 & -0.146 & -0.117 \\
                    0.296 & 0.002 & 0.296 & -0.634 & -0.046
                \end{array}
            \end{small}
        \right].
    \end{align*}
    Solving Theorem~\ref{generalized_sectored_disk_lemma} with $\gamma = 1$, $\alpha = -\pi/4$ and $\beta = \pi/3$, we have
    \begin{align*}
        X&=
        \begin{bmatrix}
            0.0055 & -0.2315 & -0.0719 \\
            -0.2315 & -0.2350 & -0.2768 \\
            -0.0719 & -0.2768 & 0.0197
        \end{bmatrix},\\
        Y&=
        \begin{bmatrix}
            0.0319 & 0.0838 & 0.0432 \\
            0.0838 & 0.1200 & 0.0851 \\
            0.0432 & 0.0851 & 0.0357
        \end{bmatrix},\\
        k_1 &= 0.8189,
        k_2 = 0.0221,
        k_3 = 0.1010,
        k_4 = 0.0580.
    \end{align*}
    Applying the bounded real lemma~\cite{liu2016robust} to this example, one can check that $G$ is not norm-bounded by $\gamma=1$. 
    Therefore, the robust closed-loop stability of $G\#\Delta$ against all $\Delta\in\mathcal{U}_\gamma(\alpha,\beta)$ cannot be ensured by the
    small gain theorem nor the small phase theorem.
\end{expl}

\section{Conclusion}\label{conclusion}
In this study, we analyzed the superset and subset of the DW shell union of sectored-disk matrices. Through this analysis, we derived both a new sufficient condition and a new necessary condition for the proposed matrix sectored-disk problem, employing DW shell separation.
Building upon these matrix results, we introduced a frequency-wise robust stability condition for LTI systems subject to mixed gain-phase uncertainty. Lastly, we developed a state-space condition based on LMIs for the computation and verification of the frequency-domain stability condition.

As for the future research, we are also interested in the following research directions.
\begin{enumerate}
    \item We aim to identify a more stringent superset of $\mathcal{DW}(\secmat)$, striving for a less conservative sufficient condition. As illustrated in Fig.~\ref{DW_shell_union_numerical}, noticeable gaps persist between the superset and the DW shell union. We anticipate that addressing these gaps will yield additional mathematical insights, potentially involving inequalities related to sectored-disk matrices. 
    \item We additionally seek to extend the applicability of the sectored-disk problem to encompass semi-stable LTI systems represented by $G$ and $\Delta$. This generalization is particularly relevant in practical applications, as exemplified in our introductory example.
    \item In addition to the sectored-disk problem, the consideration of a multiple-sectored-disk problem holds significant importance in both theoretical studies and practical applications. A comprehensive framework addressing this problem has been introduced in~\cite{Tits1999RobustnessUB}, where PS-SSV is studied as a stability index. We aspire to extend our approach to develop a more refined robust stability condition for the multiple-sectored-disk problem.
\end{enumerate}

\appendix
\section{Appendix}\label{proof_of_prop_DW_shell_union_rough_shape}
\subsection{Proof of Proposition~\ref{prop_DW_shell_union_rough_shape}}
\begin{proof}
    To prove \eqref{normal_matrix_DW_shell_union}, let $M\in\nsecmatsym$. Then from Lemma \ref{DW_shell_properties}, we can see that $\mathcal{DW}(M)$ is the convex hull of its eigenvalues. Denote the eigenvalues of $M$ by $\lambda_i, i=1,...,n$.

    Moreover, we also have $|\lambda_i|^2\leq \gamma^2$ and 
    $\angle{\lambda_i}\in[-\alpha,\alpha]$ by the phase property \cite{wang2020phases}.
    This implies $|\lambda_i|\leq \realpart{\lambda_i}\sec(\alpha)$.
    As a result, 
    the vertices of $\mathcal{DW}(M)$ satisfy 
    $(\realpart\lambda_i,\imagepart\lambda_i,|\lambda_i|^2) \in 
        \mathcal{H}_\gamma \cap
        \mathcal{K}_{\gamma\sec(\alpha)} \cap \mathcal{V}_\alpha$,
    which is the intersection of convex sets.
    This reveals that $\mathcal{DW}(M) \subset \mathcal{H}_\gamma \cap \mathcal{K}_{\gamma\sec(\alpha)} \cap \mathcal{V}_\alpha$.
    On the other hand, let
    $(x_0,y_0,z_0)\in\mathcal{H}_\gamma \cap
    \mathcal{K}_{\gamma\sec(\alpha)} \cap \mathcal{V}_\alpha$.
    We want to find $M\in\nsecmatsym$ such that $(x_0,y_0,z_0)\in\mathcal{DW}(M)$.
    
    The construction of $M$ needs to be divided into the following two cases.
    \begin{enumerate}
        \item $(x_0,y_0,z_0)$ satisfies 
        $z_0^{1/2}\cos(\alpha)\leq x_0 \leq z_0^{1/2}$.
        \item $(x_0,y_0,z_0)$ satisfies 
        $\gamma^{-1}\cos(\alpha)z_0\leq x_0 \leq z_0^{1/2}\cos(\alpha)$.
    \end{enumerate}

    For Case~1), we will find two points on the paraboloid that interpolates
    $(x_0,y_0,z_0)$.    
    Since $z_0^{1/2}\cos(\alpha)\leq x_0\leq z_0^{1/2}$, 
    there exists $\theta\in[-\alpha,\alpha]$ such that $x_0=z_0^{1/2}\cos(\theta)$.
    Construct that 
    $$M=\diag\{z_0^{1/2}e^{j\theta},z_0^{1/2}e^{-j\theta},...,z_0^{1/2}e^{-j\theta}\}.$$

    By Lemma~\ref{DW_shell_properties}, 
    $\mathcal{DW}(M)$ is the line segment between $(x_0,-x_0\tan\theta,z_0)$
    and $(x_0,x_0\tan\theta,z_0)$.
    Since $x_0^2+y_0^2\leq z_0$, we have
    $y_0\in[-z_0^{1/2}\cos\theta,z_0^{1/2}\cos(\theta)]$,
    yielding that $(x_0,y_0,z_0)\in\mathcal{DW}(M)$.

    For Case~2), 
    let
    \begin{align*}
        x_1=\frac{z_0}{x_0}\cos^2(\alpha),~y_1=\frac{y_0}{x_0}x_1,~z_1=\frac{z_0}{x_0}x_1.
    \end{align*}
    Notice that $x_1=z_1^{1/2}\cos(\alpha)$.
    By the assumption that $x_0^2\leq z_0\cos^2(\alpha)$, we have 
    $x_1\geq x_0$.
    Consequently, $(x_0,y_0,z_0)$ is on the triangle facet with vertices 
    $(0,0,0)$, $\left(x_1,z_1^{1/2}\sin(\alpha),z_1\right)$ and $\left(x_1,-z_1^{1/2}\sin(\alpha),z_1\right)$. Construct that
    $$M=\diag\{z_1^{1/2}e^{j\alpha}, z_1^{1/2}e^{-j\alpha}, 0,...,0\},$$
    whereby $(x_0,y_0,z_0)\in\mathcal{DW}(M)\subset\mathcal{DW}(\nsecmatsym)$.
    
    To prove \eqref{sectored_disk_matrix_DW_shell_union}, from the norm constraint of $\secmatsym$ or from~\cite[Theorem 2.1(a)]{Li2008DAVISWIELANDTSO}, we have $\mathcal{DW}(\secmatsym)\subset \mathcal{H}_\gamma $.
    
    From the definition of DW shell, 
    the XY-projection of the DW shell of a matrix is the numerical range of the matrix with $\mathcal{V}_\alpha$ being the angular constraint on the XY-plane. 
    It then follows $\mathcal{DW}(\secmatsym)\subset \mathcal{V}_\alpha$.
\end{proof}

\section*{Acknowledgments}
The authors would like to thank Chao Chen, Jianqi Chen, Wei Chen, Xin Mao and Ding Zhang for valuable discussions. The authors would also like to thank the anonymous reviewers for insightful comments and suggestions, which helped improve the paper.

\bibliographystyle{IEEEtran}
\bibliography{SectoredDiskProblem}
\begin{IEEEbiography}[{\includegraphics[width=1in,height=1.25in,clip,keepaspectratio]{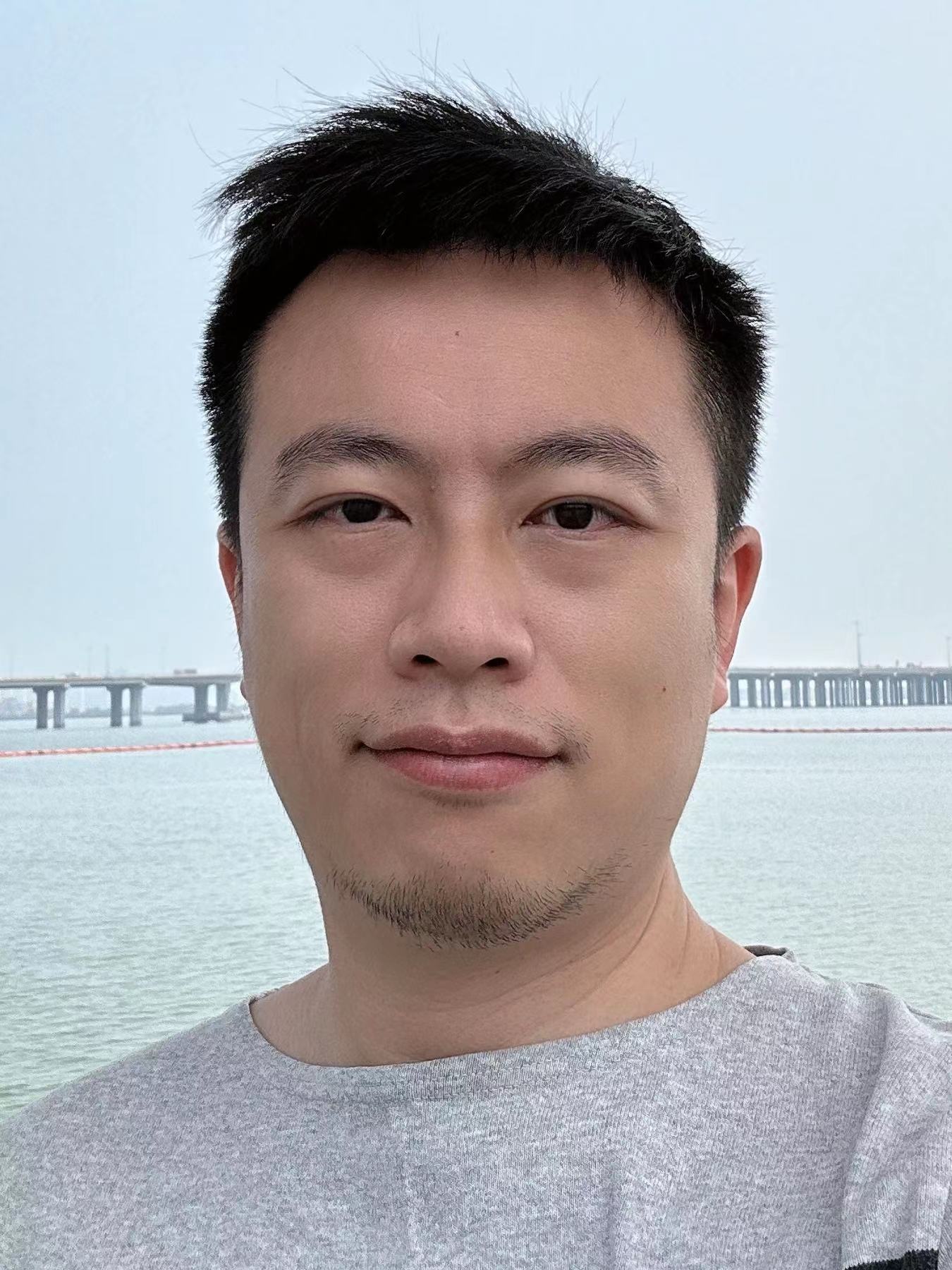}}]{Jiajin Liang} 
    received his B.A. degree in financial mathematics from South China Normal University in 2004. He received his M.Sc\ degree in operational research and control mathematics from Fudan University in 2008. He is now pursuing the Ph.D. degree in electronic and computer engineering at Hong Kong University of Science and Technology, Hong Kong SAR, China. 
    
    His research interests include systems, control, and mathematics for control theory.
   \end{IEEEbiography}
\begin{IEEEbiography}[{\includegraphics[width=1in,height=1.25in,clip,keepaspectratio]{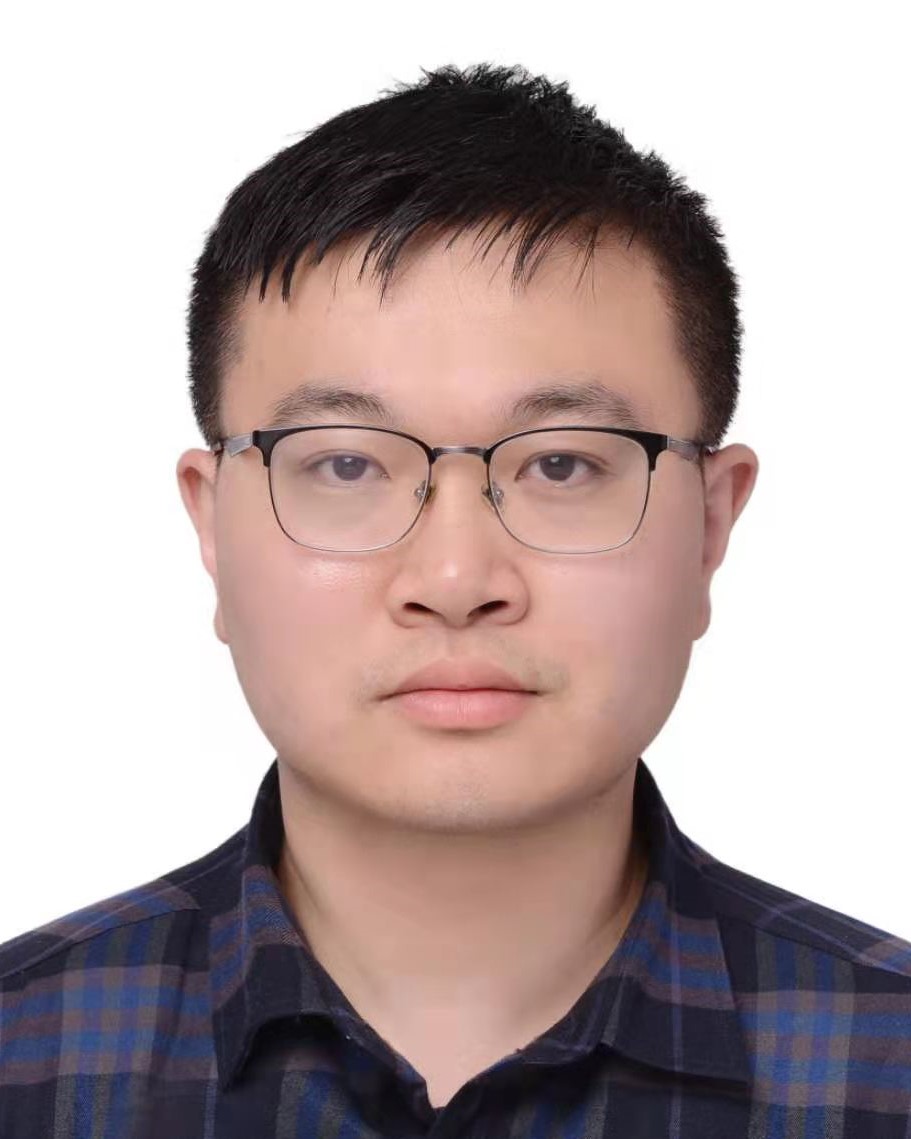}}]{Di Zhao} 
    received his B.E. degree in electronics and information engineering from Huazhong University of Science and Technology in 2014. He received his Ph.D.\ degree in electronic and computer engineering from the Hong Kong University of Science and Technology (HKUST) in 2019. From Aug. 2018 to Jan. 2019, he was a visiting student researcher at Lund University. After shortly working as a post-doctoral researcher with HKUST at the end of 2019, he was a researcher with the Cyber-Physical Systems Laboratory at HKUST Shenzhen Research Institute in 2020. He joined Tongji University at the end of 2020, where he is now an Associate Professor of Department of Control Science and Engineering. 
    
    His research interests include networked control systems, robust control, optimal control, monotone systems, nonlinear systems, phase theory and rank optimization. 
   \end{IEEEbiography}
   \begin{IEEEbiography}[{\includegraphics[width=1in,height=1.25in,clip,keepaspectratio]{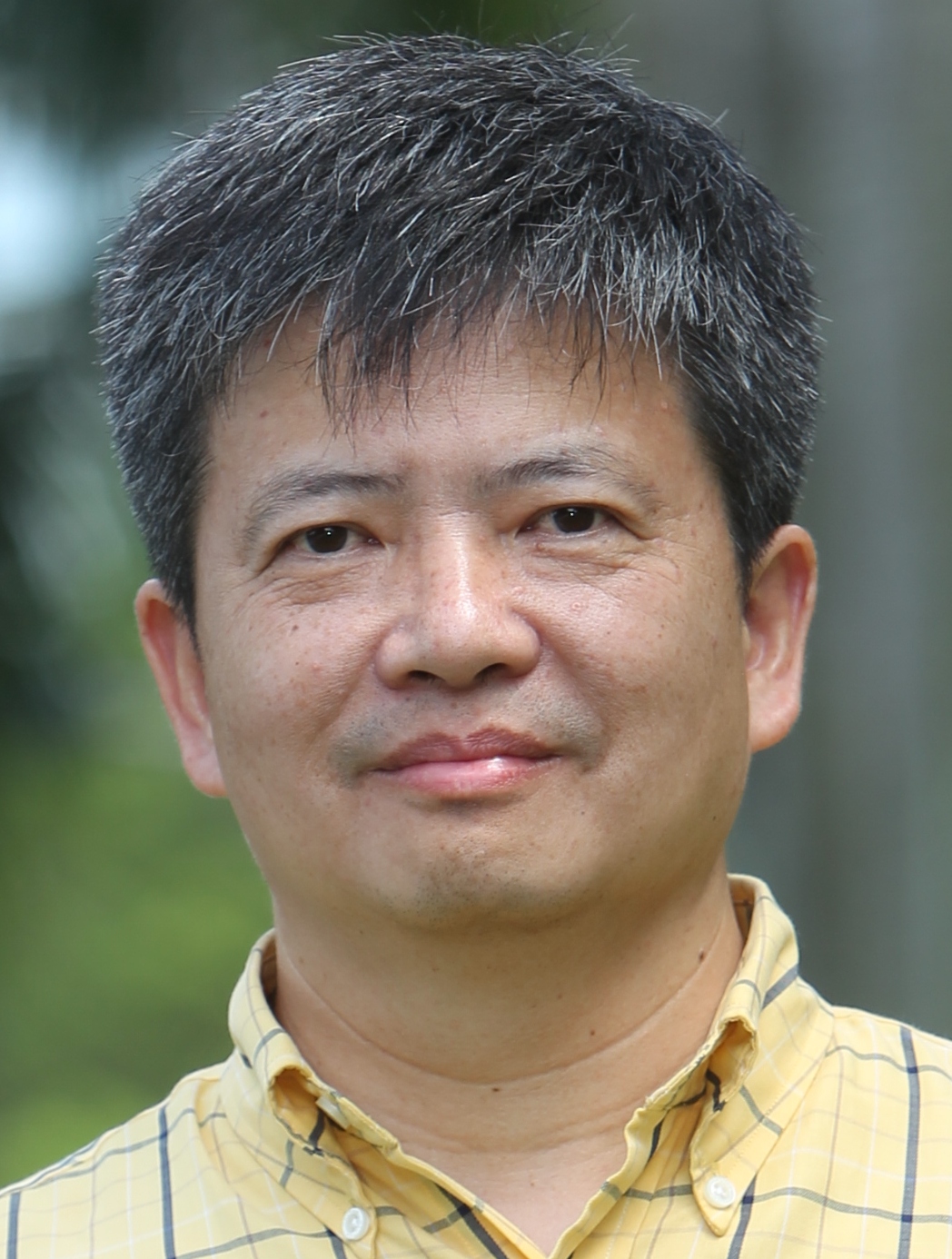}}]{Li Qiu}
    (F'07) received his Ph.D.\ degree in electrical engineering from the
    University of Toronto in 1990. After briefly working in the Canadian Space Agency, the Fields Institute for Research in Mathematical Sciences (Waterloo),
    and the Institute of Mathematics and its Applications (Minneapolis), he joined the Hong Kong
    University of Science and Technology in 1993, where he is now a
    Professor of Electronic and Computer Engineering.
    
    Prof.~Qiu's research interests include system, control,
    information theory, and mathematics for information technology, as well as their
    applications in manufacturing industry and energy systems.
    He is also interested in control education and coauthored an undergraduate textbook
    ``Introduction to Feedback Control'' which was published by Prentice-Hall in 2009.
    He served as an associate editor of the {\em IEEE Transactions on Automatic Control} and an associate editor of {\em Automatica}. He was the general chair of the 7th Asian Control Conference, which was held in Hong Kong in 2009. He was a Distinguished Lecturer from 2007 to 2010 and was a member of the Board of Governors in 2012 and 2017 of the IEEE Control Systems Society. He is a member of the steering committees of Asian Control Association (ACA) and International Symposiums of Mathematical Theory of Networks and Systems (MTNS). He is the founding chairperson of the Hong Kong Automatic Control Association, serving the term 2014-2017. He is a Fellow of IEEE and a Fellow of IFAC.
   \end{IEEEbiography}
\end{document}